\documentclass{article}
\usepackage[margin=1.25in]{geometry}
\usepackage{amsmath}
\usepackage{amssymb}
\usepackage{xypic}
\usepackage{amsthm}
\usepackage{mathtools}

\newtheorem{lemma}{Lemma}
\newtheorem*{lemma*}{Lemma}
\newtheorem{proposition}{Proposition}
\newtheorem*{proposition*}{Proposition}
\newtheorem{theorem}{Theorem}
\newtheorem{theorem*}{Theorem*}
\newtheorem{corollary}{Corollary}

\theoremstyle{definition}
\newtheorem{defi}{Definition}
\newtheorem{example}{Example}
\newtheorem{remark}{Remark}

\newcommand{\Rep}{\textnormal{Rep}}
\newcommand{\Hom}{\textnormal{Hom}}

\newcommand{\I}{\mathbb{I}}

\newcommand{\colim}{\textnormal{colim}}
\renewcommand{\lim}[1]{\underset{#1}{\textnormal{lim}}\;}

\newcommand{\R}{\mathbb{R}}
\newcommand{\Z}{\mathbb{Z}}
\newcommand{\C}{\mathbb{C}}

\newcommand{\F}{\mathbb{F}}
\newcommand{\Q}{\mathbb{Q}}

\newcommand{\g}{\mathfrak{g}}

\newcommand{\Ad}{\textnormal{Ad}}
\newcommand{\adj}{\textnormal{ad}}

\newcommand{\Aut}{\textnormal{Aut}}

\newcommand{\ad}{\textnormal{ad}}

\newcommand{\Ant}{{\mathrm{Ant}}}
\newcommand{\Mat}{{\mathrm{Mat}}}
\newcommand{\GL}{{\mathrm{GL}}}
\newcommand{\SL}{{\mathrm{SL}}}
\newcommand{\Sp}{{\mathrm{Sp}}}

\newcommand{\He}{{\mathrm{H}}}
\newcommand{\st}{\mathfrak{st}}
\newcommand{\h}{{\mathfrak{h}}}
\newcommand{\Gr}{{\mathrm{Gr}}}
\newcommand{\can}{{\mathrm{can}}}
\renewcommand{\O}{\mathrm{O}}
\newcommand{\U}{\mathrm{U}}
\newcommand{\Or}{\mathrm{Or}}
\renewcommand{\I}{\mathrm{Iso}}

\newcommand{\Emb}{\mathrm{Emb}}
\newcommand{\im}{\mathrm{im}}

\newcommand{\Sym}{\mathrm{Sym}}

\title{Topological components of spaces of commuting elements in connected nilpotent Lie groups}

\author{Omar Antol\'in Camarena and Bernardo Villarreal}

\begin{document}

\maketitle

\def\lf{\left\lfloor}
\def\rf{\right\rfloor}

\maketitle

\begin{abstract}
We study the homotopy type of spaces of commuting elements in connected nilpotent Lie groups, via almost commuting elements in their Lie algebras. We give a necessary and sufficient condition on the fundamental group of such a Lie group $G$ to ensure $\Hom(\Z^k,G)$ is path-connected. In particular for the reduced upper unitriangular groups and the reduced generalized Heisenberg groups, $\Hom(\Z^k,G)$ is not path-connected, and we compute the homotopy type of its path-connected components in terms of Stiefel manifolds and the maximal torus of $G$.  
\end{abstract}

\section{Introduction}


The space \(\Hom (\Z^k,G)\) of ordered commuting $k$-tuples in a Lie group \(G\) has been studied in various settings such as geometry, homotopy theory and mathematical physics ---in large part because \(\Hom (\Z^k,G)\) may be identified with the space of flat \(G\)-bundles over a rank \(k\)-torus. From the point of view of physics, the case of interest is when \(G\) is a compact Lie group, and from the homotopical point of view this may be justified by a remarkable result from A. Pettet and J. Souto \cite{P-S} which states that for the group of complex or real points of a reductive algebraic group \(G\) (defined over \(\R\) in the latter case), the inclusion \(K \hookrightarrow G\) of its maximal compact subgroup induces a deformation retraction \(\Hom (\Z^k,G)\simeq\Hom(\Z^k,K)\). However, there are many Lie groups which are not algebraic where such a homotopy equivalence is far from achievable. For instance, W. Goldman \cite{GoldmanFlat2} studied the space of commuting pairs in  \(\widetilde{\He}_3(\R)\), the reduced Heisenberg group over $\R$, and showed that \(\Hom(\Z^2,\widetilde{\He}_3(\R))\) has an infinite number of path-connected components, whereas the space of commuting pairs in $S^1\subset \widetilde{\He}_3(\R)$ is path-connected. Moreover, Goldman shows that at the level of path-connected components, the map that forgets the flat connection
\[\pi_0(\Hom(\Z^k,G))\to [B\Z^k,BG] = H^2(\Z^k;\pi_1(G))\,,\]  
specialized to $G=\widetilde{\He}_3(\R)$ and $k=2$, is bijective. In this paper we will use this map to analyze the homotopy type of $\Hom(\Z^k,G)$ for connected nilpotent Lie groups, but from the Lie algebra perspective which is viable under these assumptions. 

For a large class of nilpotent groups, we obtain a description of $\Hom(\Z^k,G)$ in terms of Stiefel manifolds and its maximal compact subgroup. Consider the group of $n\times n$ upper unitriangular matrices $\U_{n}(\F)$ over $\F$, where $\F = \R$ or $\C$. We let $A\subset \U_{n}(\F)$ denote the standard central subgroup isomorphic to $\Z$ in the real case and to $\Z\times\Z$ in the complex case. Then we define the \emph{reduced} group of $n\times n$ upper unitriangular matrices as
\[\widetilde{\U}_n(\F):=\U_n(\F)/A\,.\] 
For $n\geq 2$, the maximal compact subgroup $K\subset \widetilde{\U}_n(\F)$ is in all real cases the central subgroup $S^1$ and in all complex cases the central subgroup $S^1\times S^1$. In particular the space of commuting $k$-tuples of $K$ is path-connected for every $k\geq 1$. The situation for the ambient groups is quite different. Our first result concerns a complete identification of the homotopy type of $\Hom(\Z^k,\widetilde{\U}_n(\F))$. For convenience we set \(D(\R):=\C\) and \(D(\C):=\mathbb{H}\) (the quaternions), so that we have a unified way to refer to them as \(D(\F)\) ---here \(D\) stands for the Cayley--Dickson doubling construction. 
\def\lf{\left\lfloor}
\def\rf{\right\rfloor}

\begin{theorem}\label{main-theoremST}
Let \(k,n\geq 1\), and \(\F=\R\) or \(\C\). Then the map that forgets the flat structure at the level of path-connected components
\[\pi_0(\Hom(\Z^k,\widetilde{\U}_{n+2}(\F))\to H^2(\Z^k;A)\]
is injective for every \(k\geq 2\), and surjective if and only if \(\lf\frac{k}{2}\rf\leq n\). Moreover, each component parametrized by a class \(\beta\in H^2(\Z^k;A)\) is homotopy equivalent to \[V_r(D(\F)^n)\times (S^1)^{k\dim_\R\F }\,,\] where \(2r=\mathrm{rank}(\beta)\).
\end{theorem}

\noindent Here $V_q(D(\F)^m)$ as usual stands for the Stiefel manifold of $q$-frames in $D(\F)^m$, and we think of a cohomology class $\beta$ as an antisymmetric bilinear form over $\F^k$ where $\mathrm{rank}(\beta)$ is the rank of its associated matrix with respect to the canonical basis of $\F^k$. 

The generalized Heisenberg groups $\He_{2n+1}(\F)$, of dimension $2n+1$ over $\F$, are naturally subgroups of $\U_{n+2}(\F)$ that contain the center, so we analogously define the \emph{reduced} generalized Heisenberg group of dimension $2n+1$ as \[\widetilde{\He}_{2n+1}(\F): = {\He}_{2n+1}(\F)/A\,.\]  Our theorem heavily relies on the fact that the inclusion (Proposition \ref{prop:defretrank1}) 
 \[\Hom(\Z^k,\widetilde{\He}_{2n+1}(\F))\xhookrightarrow{\simeq}\Hom(\Z^k,\widetilde{\U}_{n+2}(\F))\]
admits a deformation retraction. Then we show that, up to homotopy, the path-connected components of $\Hom(\Z^k,\widetilde{\He}_{2n+1}(\F))$ correspond to \emph{symplectic} frames of $(\F^{2n},\omega_n)$ where $\omega_n$ is the symplectic form that determines $\He_{2n+1}(\F)$. In Subsection \ref{sec:looped-comm-map} we give an interesting application of Theorem \ref{main-theoremST} to classifying spaces for commutativity. We exhibit a connected Lie group $G$ (infinite dimensional) whose looped commutator map splits, up to homotopy.

Our approach to understand the homotopy type of commuting elements in a connected nilpotent Lie group $G$ is via \emph{almost} commuting elements in its Lie algebra $\g$. These spaces are related by the Baker--Campbell--Hausdorff formula (see Proposition \ref{prop:AC-LieGrp-LieAlg} for a precise statement). When $G$ is simply-connected and nilpotent we obtain that $\Hom(\Z^k,G)$ is contractible (Corollary \ref{Cor:HomSimpConn}), and hence it has the same homotopy type of the space of commuting elements in its maximal compact subgroup $K=\{1\}$, the trivial subgroup. If $G$ is connected and nilpotent but with non-trivial $\pi_1(G)$, a deformation retraction to the subspace of commuting elements in its maximal compact subgroup may or may not exist depending on a condition on $\pi_1(G)$, when regarded as a discrete subgroup of the center of the Lie algebra $Z(\g)$, namely \(\pi_1(G) \cong \ker(\exp|_{Z(\g)})\). We merge our results in Corollary \ref{cor:def-ret-1-comp} and Corollary \ref{cor:hom-path-conn} into the following theorem.

\begin{theorem}
Let $G$ be a connected nilpotent Lie group and $K\subset G$ a maximal compact subgroup. Then the inclusion $ \Hom(\Z^k,K)\hookrightarrow\Hom(\Z^k,G)$ admits a deformation retraction if and only if no non-trivial element in $\pi_1(G)$ corresponds to a Lie bracket under the isomorphism \(\pi_1(G) \cong \ker(\exp|_{Z(\g)})\).
\end{theorem}

As it  is well known, for non-compact Lie groups $G$, the space of isomorphism classes of representations \(\Rep(\Z^k,G)=\Hom(\Z^k,G)/G\) in general does not behave nicely from the topological point of view, as it may not even be a \(\mathcal{T}_1\)-space. However, Theorem \ref{main-theoremST} still gives a full account on the indexing set of path-connected components for $\Rep(\Z^k,\widetilde{\U}_{n+2}(\F))$ by a standard argument of injective composites (see Corollary \ref{cor:rep-redHeis}).


For $\widetilde{\He}_{2n+1}(\F)$ we describe in detail, what comes into place when modding out by the conjugation action. For instance, it turns out the the representation space $\Rep(\Z^2,\widetilde{\He}_{2n+1}(\F))$ contains the so-called \emph{irrational tori} (see Example \ref{example-irrational-tori}). It is not hard to believe then, that from the homotopical point of view, studying these moduli spaces   as they stand may not be achievable. As an alternative, we compute the maximal Hausdorff quotient which we denote by \(\Rep(\Z^k,\widetilde{\He}_{2n+1}(\F))^{\mathcal{T}_2}\), and show that the path-connected components pa\-ra\-me\-trized by a form \(\beta\in H^2(\Z^k,A)\) coincide with its maximal Hausdorff subspaces
\[\{f\in\Hom(\F^k,\F^{2n})\,|\,f^*\omega_n = \beta\}\]
 (which have the homotopy type of $V_r(D(\F)^n)$, where $r=\mathrm{rank}(\beta)$), except for the path-connected component of the trivial representation, which up to homotopy is simply $(S^1)^{k\dim_\R \F}$.

In the last part of the paper, from an independent interest, we study the space of $2n\times k$ matrices whose columns span an \emph{isotropic} subspace of $\F^{2n}$ with respect to the standard symplectic form.  In our notation this space corresponds to $X(k,n,0)$. Our motivation is that $X(k,n,0)$ is very closely related to $\Hom(\Z^k,\widetilde{\He}_{2n+1}(\F))_1$. We consider the filtration of $X(k,n,0)$ induced by the rank of a matrix. We found particularly interesting that the aforementioned filtration still carries homotopical information regarding the full space $\Hom(\Z^k,\widetilde{\He}_{2n+1}(\F))$, at least up to a fiber sequence. We denote the strata of the rank filtration by \(X(k,n,0)_d\), and then we have a smooth fiber bundle 
\[ X(k,n,0)_d\to \Gr(d,\F^k)\]
with fibers the space of linear isotropic embeddings \(\F^d\to \F^{2n}\), which we denote by $\I(2n,d,\F)$. We show that \(\I(2n,d,\F)\) is homotopy equivalent to the Stiefel manifold \(V_d(D(\F)^n)\). The main result of this last section is a homotopy stable splitting of the strata $X(k,n,0)_d$ into a wedge of Thom spaces following Crabb and Gon\c{c}alves \cite{CrabbGoncalves}. The splitting is depicted in Theorem \ref{prop:stable-splitting}.

\subsection*{Acknowledgments}

The second author would like to thank Jos\'e Cantarero, Jos\'e Manuel G\'omez and Simon Gritschacher for very helpful conversations.

\section{Almost commuting elements in nilpotent Lie groups}

We start with some structural results for commuting elements in connected nilpotent Lie groups, which are completely deducted from almost commuting elements in its Lie algebra with respect to the Lie bracket.

\subsection{A cohomological bound on connected components}

We will describe a general strategy to study almost commuting matrices, which mostly follows previously developed ideas as in \cite[Section 3]{AdemCohenGomezMPCPS}. 

Let \(G\) be a connected Lie group and \(X\) a connected CW complex. Consider a principal \(G\)-bundle \(E\to X\), with classifying map \(p\colon X\to BG\). Whenever \(E\) and \(X\) are manifolds and \(E\) has a flat connection, then \(p\) induces a group homomorphism \(p_*\colon \pi_1(X)\to G\). In general we may consider those principal \(G\)-bundles over \(X\) whose classifying map factors through \(BG^\delta\to BG\), where \(G^\delta\) is \(G\) with the discrete topology. Hence we can identify the set of group homomorphisms \(\Hom(\pi_1(X),G)\) with a subspace of \(\text{Map}(B\pi_1(X),BG)\).
 
We will be interested in parametrizing the path-connected components of \(\Hom(\pi_1(X),G)\). To do this let us consider the map
\[\Hom(\pi_1(X),G)\to H^2(X;\pi_1(G))\] 
that associates to each group homomorphism \(\rho\colon \pi_1(X)\to G\) an element in \(\tilde\rho\in H^2(X;\pi_1(G))\) by first considering the composite
\[\Hom(\pi_1(X),G)\hookrightarrow\text{Map}_*(B\pi_1(X),BG)\xrightarrow{{p_2}_*}\text{Map}_*(B\pi_1(X),B^2\pi_1(G))\xrightarrow{\pi_0}[B\pi_1(X),B^2\pi_1(G)] \,,\]
where \(p_2\colon BG\to B^2\pi_1(G)\) is the second Postnikov section of \(BG\), and then compose it with 
\[[B\pi_1(X),B^2\pi_1(G)]\xrightarrow{\cong}H^2(B\pi_1(X);\pi_1(G))\xrightarrow{p_1^*}H^2(X;\pi_1(G)),\]
where \(p_1\colon X\to B\pi_1(X)\) is the first Postnikov section of \(X\). In particular the correspondence \(\rho\mapsto \tilde\rho\) is locally constant. Let us spell out precisely what the assignment does. A map \(B\pi_1(X)\to B^2\pi_1(G)\) that factors through \(p_2\) induces a central extension \(1\to\pi_1(G)\to \Gamma\to \pi_1(X)\to 1\), which is classified by the homotopy pullback 
\[
\xymatrix{
B\Gamma\ar[r]\ar[d]&B\tilde{G}\ar[d]\ar[r]&{*}\ar[d]\\
B\pi_1(X)\ar[r]&BG\ar[r]&B^2\pi_1(G)\,,
}
\]
where \(\tilde{G}\) is a universal cover of \(G\). Then \(\tilde{\rho} = p_1^*(c)\), where \(c\) is the class of a cocycle associated to the central extension \(\Gamma\).

We will specialize to \(X = T_k\), a torus of rank \(k\). Let \(\rho\colon \Z^k\to G\) be a group homomorphism, with assigned extension \(0\to \pi_1(G) \to \Gamma \to \Z^k\to 0\) as described above. We can obtain \(\Gamma\) via the pullback diagram
\[\xymatrix{
\Gamma\ar[d]\ar[r]&\tilde{G}\ar[d]\\
\Z^k\ar[r]^\rho&G\,.}
\]
Let \(e_1,...,e_k\) be the canonical basis of \(\Z^k\), and choose lifts \(\tilde{e}_1,...,\tilde{e}_k\) to \(\Gamma\). Let \(\Ant_{k}(S)\) be the group of \(k\times k\) antisymmetric matrices with entries in an abelian group \(S\). Then the above correspondence is 
\[\rho\mapsto \tilde{\rho} \in H^2(T_k;\pi_1(G))\cong \Ant_k(\pi_1(G))\,,\]
where \(\tilde{\rho}_{ij} = [\tilde{e}_i,\tilde{e}_j]\) and  \([x,y]=x^{-1}y^{-1}xy\) is the commutator of \(\Gamma\). Indeed, if \(c\colon \Z^k\times\Z^k\to \pi_1(G)\) is the cocycle associated to the extension \(\Gamma\), then its corresponding value \(\tilde{c}\in\Ant_k(\pi_1(G))\) is \(\tilde{c}_{ij} = c(e_i,e_j)-c(e_j,e_i)\) for every \(i,j\leq k\). The function \(c(a,b)-c(b,a)\) is the antisymmetrization of \(c\), which can be seen is equal to \([\tilde{a},\tilde{b}]\) where \(\tilde{a},\tilde{b}\in \Gamma\) are lifts of \(a,b\).

The non-zero values of this correspondence give a lower bound on the number of path-connected components of \(\Hom(\Z^k,G)\). To get such bound, we will study the space of almost commuting tuples in \(\tilde{G}\), that is, for a central subgroup \(Z\subset \tilde{G}\) we will consider
\[B_k(\tilde{G},Z) = \{(x_1,...,x_k)\in\tilde{G}^k\;|\;[x_i,x_j]\in Z\text{ for all }i,j\}\,,\]
where \(Z\) is the discrete subgroup such that \(\tilde{G}/Z\cong G\), and in particular \(Z\cong\pi_1(G)\). The iterated product of the universal cover \(\pi\colon \tilde{G}\to G\) yields the commutative square
\[
\xymatrix{B_k(\tilde{G},Z)\ar@{^(->}[r]\ar[d]_{\pi^k|}&\tilde{G}^k\ar[d]^{\pi^k}\\
\Hom(\Z^k,G)\ar@{^(->}[r]&G^k\,,}
\]
and it is not hard to see that it is in fact a pullback square. Hence the left vertical arrow, \({\pi^k}|\), is a principal \(Z^k\)-bundle. The connected components \(\Hom(\Z^k,G)\) will be quotients by \(Z^k\) of components of \(B_k(\tilde{G},Z)\), but in general there is no 1--1 correspondence. But notice that the map of ordered commutators defined by
\begin{align*}
\psi\colon B_k(\tilde{G},Z)&\to \Ant_k(Z)\\ 
(x_1,...,x_k)&\mapsto ([x_i,x_j])
\end{align*}
factors as \(\pi^k|\) followed by the assignment \(\rho\to\tilde{\rho}\). Thus, if the composite 
 \[B_k(\tilde{G},Z)\xrightarrow{\pi^k|}\Hom(\Z^k,G)\to \Ant_k(Z)\,,\] 
is injective in \(\pi_0\), the components of \(\Hom(\Z^k,G)\) will be precisely the ones of \(B_k(\tilde{G},Z)\) modulo \(Z^k\). As $Z$ is a central subgroup of $\widetilde{G}$ and $\pi^k|$ is a covering map, we also have that the assignment \(\rho\to\tilde{\rho}\) is invariant under conjugation by elements of $G$. Then we also have a well defined map for the representation space
\[\Rep(\Z^k,G)\to H^2(\Z^k,\pi_1(G))\,.\]

\subsection{Linearization of almost commuting elements in nilpotent Lie groups}

Suppose \(G\) is a simply-connected nilpotent Lie group over $\R$ or over $\C$, with Lie algebra \(\g\). It is well known that in this situation the exponential map \(\exp\colon G\to\g\) is a diffeomorphism \cite[Theorem 1.127]{KnappLieGroups}. Thus we may therefore assume that every element in \(G\) is of the form \(\exp(x)\) for a unique \(x\) in \(\g\).

\begin{lemma}\label{lem:center-center}
Let \(G\) be a simply-connected nilpotent Lie group with Lie algebra \(\g\). Then the restriction
 \[\exp|_{Z(\g)}\colon Z(\g) \xrightarrow{\cong} Z(G)\] 
 is an isomorphism of abelian groups.
\end{lemma}
\begin{proof}
By the exponential map properties, it is clear that  the restriction of the exponential map to \(Z(\g)\) is an injective homomorphism. Since \(\exp\) is surjective we have \(\exp(Z(\g))\subset Z(G)\). To see that \(\exp|_{Z(\g)}\) surjects onto \(Z(G)\), first note that for every \(x,y\in \g\), the Baker-Campbell-Hausdorff formula (BCH) implies that
 \[[\exp (x),\exp (y)] = \exp([x,y]+z),\] 
where \(z\) is a finite linear combination of nested brackets in \(x\) and \(y\), each of them of length at least 3. Suppose \(\exp(x)\in Z(G)\). Then \( [\exp (x),\exp (y)] = 1\) and thus \([x,y] = -z\). We may conclude that if \(L := \langle x,y \rangle \) is the Lie subalgebra of \(\g\) generated by \(x\) and \(y\), then \([x,y]\in [L,[L,L]]\), but since \([x,y]\) generates \([L,L]\) we further obtain that \([L,[L,L]]=[L,L]\). Now \(L\) is a subalgebra of a nilpotent algebra which makes it also nilpotent, and then the previous equality implies \([L,L]=0\). As this holds for any \(y\in \g\), it follows that \(x\) is central in \(\g\).
\end{proof}

For a Lie algebra \(L\), we may consider the space of almost commuting elements with respect to the Lie bracket. That is, for an abelian subgroup of the center of \(L\), \(A\subset Z(L)\), we define
\[B_k(L,A) = \{(x_1,...,x_k)\in L^k : [x_i,x_j]\in A\text{ for every }i,j\}\,.\]

Lemma \ref{lem:center-center} tells us that for any central subgroup \(A\subset Z(\g)\) the exponential map is an \(A^k\)-equivariant diffeomorphism \(\g^k\to G^k\) where the \(A^k\)-action on \(G^k\) is the one induced by the isomorphism \(\exp|_{Z(\g)}\). This will allow us to stablish an easy dictionary from almost commuting tuples in Lie algebras to Lie groups.

\begin{proposition}\label{prop:AC-LieGrp-LieAlg}
Let \(G\) be a simply-connected nilpotent Lie group with Lie algebra \(\g\). Let \(A\subset Z(\g)\) be an abelian subgroup. Then the exponential map induces an \(A^k\)-equivariant homeomorphism
\[B_k(\g,A) \xrightarrow[\exp]{\cong} B_k(G,\exp(A))\,.\]
\end{proposition}
\begin{proof}

Let \(K=\exp(A)\). It suffices to show that under these hypotheses a bracket \([x,y]\in A\) if and only if the commutator \([\exp (x),\exp(y)]\in K\).

First, let us assume that \([x,y]\in A\). Then BCH reads as \(\exp (x) \exp (y) = \exp {(x + y + [x,y]/2)}\), and in particular \([\exp (x),\exp (y)]=\exp {([x,y])} \in K\). Now suppose \([\exp (x),\exp (y)]\in K\). We already know that \([\exp (x),\exp (y)] = \exp({[x,y] + z})\) where \(z\) is a finite combination of nested brackets in \(x\) and \(y\), each of them of length at least 3, and by Lemma \ref{lem:center-center} it follows that \([x,y] + z\in A\). We now proceed similarly as in the proof of Lemma \ref{lem:center-center}: let \(L := \langle x,y \rangle \) be the Lie subalgebra of \(\g\) generated by \(x\) and \(y\). Applying the operators \([x,-]\) and \([y,-]\) to \([x,y]+z\) we obtain that \([x,[x,y]]\) and \([y,[x,y]]\) are in \([L,[L,[L,L]]]]\), so that \([L,[L,[L,L]]]=[L,[L,L]]\). As \(L\) is nilpotent we further obtain that \([L,L]\) is 2-step nilpotent. In particular all nested brackets in \(x\) and \(y\) of length at least 3 vanish, hence \(z=0\). It follows that \(\exp([x,y])\in K\), as desired.
\end{proof}

In particular we have a homeomorphism 
\begin{align}
B_k(\g,A)/A^k\xrightarrow[\exp]{\cong}B_k(G,\exp(A))/A^k\cong \Hom(\Z^k,G/\exp(A))\,.\label{hom&ComElLieAlg}
\end{align}

\begin{corollary}\label{Cor:HomSimpConn}
Let \(G\) be a simply-connected nilpotent Lie group. Then \(\Hom(\Z^k,G)\) is contractible for every \(k\).
\end{corollary}
\begin{proof}
Specializing Proposition \ref{prop:AC-LieGrp-LieAlg} to \(A=\{0\}\), on one side we obtain \(\Hom(\Z^k,G)\) and on the other the space \(B_k(\g,\{0\})\) which is star-convex, hence contractible.  
\end{proof}

\begin{remark}
Corollary \ref{Cor:HomSimpConn} and the equivalence of categories between finite dimensional nilpotent Lie algebras and unipotent algebraic groups over a field of characteristic 0 \cite[IV-Corollaire 4.5]{DemazureGabriel}, expands on Pettet and Souto's deformation retraction, to include unipotent algebraic groups over \(\R\) or \(\C\), as in all these algebraic groups the maximal compact subgroup of the group of complex or real points is the trivial subgroup.
\end{remark}

\begin{remark}\label{rem:G-action-on-ComLieAlg}
A connected nilpotent Lie group can always be expressed as a quotient \(G/\exp(A)\), where \(G\) is a universal cover and \(A\) is a discrete central subgroup of \(\g\). Recall that the representation space \(\Rep(\Z^k,G/\exp(A))\) is the orbit space of \(\Hom(\Z^k,G/\exp(A)) \) under the conjugation action of \(G/\exp(A)\), but since conjugation by elements of \(\exp(A)\) is the identity, it is the same orbit space under the full conjugation action of \(G\). Let \(\Ad_g\colon G\to G\) be the conjugation map by \(g\in G\), and let \(\ad(x) = [x,-]\colon \g\to\g\). Since the exponential map is natural one can verify that 
\[\Ad_{\exp(x)}(\exp(y)) = \exp\left(\sum_{n=0}^\infty\ad(x)^n(y)/n!\right)\]
(in fact since \(\g\) is nilpotent the sum is actually finite). We can define a linear \(G\)-action on \(\g\) via 
\begin{align}
\exp(x)\cdot y := \sum_{n=0}^\infty\ad(x)^n(y)/n!\,.\label{G-action on Lie}
\end{align} 
Under this action the \(A^k\)-equivariant homeomorphism of Proposition \ref{prop:AC-LieGrp-LieAlg} becomes \(G\)-equivariant. Now, as \(A\) acts additively,  \(G\) acts linearly over \(\g\) and trivially over \(A\), we conclude that the actions commute, so at the level of orbit spaces we have homeomorphisms 
\begin{align}
(B_k(\g,A)/G)/A^k\cong(B_k(\g,A)/A^k)/G\xrightarrow[\exp]{\cong}\Rep(\Z^k,G/\exp(A))\,.\label{rep&ComElLieAlg}
\end{align}
 An immediate consequence of the linearity of the \(G\)-action over its Lie algebra \(\g\) is that the contraction of Corollary \ref{Cor:HomSimpConn} is preserved after taking quotients, as it is attained via a dilation. Hence
 \(\Rep(\Z^k,G)\) is contractible, as well. 
\end{remark}

\subsection{Almost commuting elements in \(\st_{n+2}(\F)\)}

The proof of Proposition \ref{prop:AC-LieGrp-LieAlg} already suggests that commuting elements are represented by 2-step nilpotent Lie subalgebras whose center lies in the center of the Lie algebra. When the values of the Lie bracket of all these Lie subalgebras depend on a common subalgebra, one may be able to deform them into a single one. This is the case of the Lie algebra of strict upper triangular matrices.

Let \(\st_m(\F)\) denote the Lie algebra of strict upper triangular \(m\times m\) matrices with coefficients in \(\F\). We can write a matrix in \(\st_{n+2}(\F)\) as
\begin{align}
\st_{n+2}(a,b,x,z): = \left(
\begin{matrix}
0 & a & &z\\
& \ddots  & x &\\
&& \ddots & b \\
& &&0
\end{matrix}
\right)\label{strict-upper}
\end{align}
where \(a\in \Mat_{1\times n}(\F)\), \(b\in \Mat_{n\times 1}(\F)\), \(x\in \st_n(\F)\) and \(z\in \F\). In this notation the center of \(\st_n(\F)\) is precisely 
\[\{\st_{n+2}(0,0,0,z):z\in\F\}\cong \F.\]
 Let \(\h_{2n+1}(\F)\) denote the \(2n+1\) dimensional Heisenberg algebra over \(\F\). As a Lie subalgebra of \(\st_{n+2}(\F)\), we can write \(\h_{2n+1}(\F)\) as the matrices in \(\st_{n+2}(\F)\) with coordinate \(x=0\). Under this identification we have \(Z(\h_{2n+1}(\F)) = Z(\st_n(\F))\).

\begin{proposition}\label{prop:defretrank1}
Let \(A\subset \F\) be an abelian subgroup. Then the inclusion 
\[B_k(\h_{2n+1}(\F),A)\xhookrightarrow{\simeq} B_k(\st_{n+2}(\F),A)\] admits an \(A^k\)-equivariant deformation retract.
\end{proposition}

\begin{proof}
Define \(H\colon \st_{n+2}(\F)^k\times I\to \st_{n+2}(\F)^k\) as the linear homotopy that multiplies each coordinate \(x\) of the \(k\)-tuple by \(t\), for every \(t\in I\), and is the identity elsewhere. This is the obvious linear retraction of \(\st_{n+2}(\F)\) into \(\h_{2n+1}(\F)\). Now given a pair of matrices \(\st_{n+2}(a,b,x,z)\) and \(\st_{n+2}(a^\prime,b^\prime,x^\prime,z^\prime)\) in \(\st_{n+2}(\F)\), its Lie bracket is given by
\begin{align}
[\st_{n+2}(a,b,x,z),\st_{n+2}(a^\prime,b^\prime,x^\prime,z^\prime)] = \st_{n+2}(ax^\prime-a^\prime x,x^\prime b-xb^\prime,[x,x^\prime],ab^\prime-a^\prime b)\,.\label{bracket-sut}
\end{align}
 But notice that if (\ref{bracket-sut}) lies in \(A\subset Z(\st_{n+2}(\F))\), then 
\begin{align*}
ax^\prime-a^\prime x &= 0\\
x^\prime b-xb^\prime &= 0\\
[x,x^\prime] &= 0
\end{align*}
and \(ab^\prime-a^\prime b\in A\). Clearly all these expressions are preserved by the linear retraction, and therefore the restriction of \(H\) to \(B_k(\st_{n+2}(\F),A)\times I\) gives the desired deformation retraction. It remains to verify that \(H\) is \(A^k\)-equivariant, but this is clear since the linear retraction is the identity over the coordinate of the center.
\end{proof}

 Let \(\widetilde{\U}_{n}(\F)\) be the reduced group of $n\times n$ upper unitriangular matrices and let \(\widetilde{\He}_{2n+1}(\F)\) be the reduced Heisenberg group of dimension $2n+1$ defined in the introduction. Proposition \ref{prop:defretrank1} and naturality of homeomorphism (\ref{hom&ComElLieAlg}) yield:

\begin{corollary}\label{cor: hom-equiv-comm-elem-Hien-Upper}
 The inclusion 
\[\Hom(\Z^k,\widetilde{\He}_{2n+1}(\F))\xhookrightarrow{\simeq}\Hom(\Z^k,\widetilde{\U}_{n+2}(\F))\]
admits a deformation retract.
\end{corollary}

\subsection{A decomposition for almost commuting elements in finite dimensional Lie algebras}

If \(G\) is a simply-connected Lie group with Lie algebra \(\g\) we may deduce from the first part of the proof of Proposition \ref{prop:AC-LieGrp-LieAlg} that there is a commutative diagram
\[
\xymatrix{
B_k(\g,A)\ar[r]^-{\exp^k|}\ar[d]_{\psi^{\prime}}&B_k(G,\exp(A))\ar[d]^\psi\\
\Ant_k(A)\ar[r]^-{\exp^{k^2}|}&\Ant_k(\exp(A))
}
\]
where \(\psi^\prime(x_1,...,x_k) = ([x_i,x_j])\). 

We can also make sense of the map \(\psi^{\prime}\colon B_k(\g,A)\to\Ant_k(A)\) for an arbitrary finite dimensional Lie algebra $\g$ over \(\F\). First notice that it is the restriction of a map \(\g^k\to \Ant_k([\g,\g])\), and consider the non-degenerate alternating bilinear map \(\omega\colon(\g/Z(\g))\times(\g/Z(\g))\to [\g,\g]\) induced by the Lie bracket on \(\g\). Then we have a map
\begin{align}
\varphi\colon(\g/Z(\g))^k&\to \Ant_k([\g,\g])\label{varphi}\\
(\overline{x}_1,...,\overline{x}_k)&\mapsto (\omega(x_i,x_j))\nonumber\,,
\end{align} 
which gives the following splitting.

\begin{lemma}\label{upper-bound-AlmCom}
Let \(A\subset Z(\g)\cap[\g,\g]\). Consider the \(A^k\)-actions over \(B_k(\g,A)\) given by addition on each coordinate, and over \((\g/Z(\g))^k\times(Z(\g))^k\) given by the trivial action on the first factor and by addition on the second factor. Then there is an \(A^k\)-equivariant homeomorphism 
 \[B_k(\g,A)\xrightarrow{\cong}\varphi^{-1}(\Ant_{k}(A))\times(Z(\g))^k\,.\]
\end{lemma}
\begin{proof}
Given an inner product in \(\g\), the orthogonal projections \(p\colon \g\to Z(\g)\) and \(q\colon \g\to \ker p\) induce a decomposition \(\ \g \xrightarrow{\cong}(\g/Z(\g))\times Z(\g)\) that takes \(x+a\mapsto (q(x), p(x)+a)\) for every  \(a\in Z(\g)\). Hence it induces an \(A^k\)-equivariant linear isomorphism  \(\g^k\xrightarrow{\cong} (\g/Z(\g))^k\times Z(\g)^k\). The restriction to \(B_k(\g,A)\) yields the desired isomorphism.
\end{proof}

We would like to know if the non-empty fibers \(\varphi^{-1}(B)\) over \(B\in\Ant_k(A)\) are connected, in order to describe the path-connected components of \(B_k(\g,A)\). Let us describe the map \(\varphi\) in a little more generality. Suppose \(V\) and \(Z\) are vector spaces over \(\F\), and let \(\omega\colon V\times V \to Z\) be a skew-symmetric bilinear map. We say that \(\omega\) is \emph{non-degenerate} if its adjoint map \(\omega^{\adj}\colon V\to \Hom(V,Z)\) is injective. Note that since \(\omega\) is skew-symmetric, the injectivity of \(\omega^{\adj}\) does not depend on whether we choose the right or left evaluation. Let \(\Omega^{\mathrm{skew}}(U,Z)\) be the space of skew-symmetric bilinear maps \(\beta\colon U\times U\to Z\). Consider the map
\begin{align}
\varphi_{U,V}\colon\Hom(U,V)&\to \Omega^{\mathrm{skew}}(U,Z)\label{varphi_U}\\
f&\mapsto f^*\omega\,.\nonumber
\end{align} 
A matrix \(B\in \Ant_k(Z)\) is equivalent to a skew-symmetric bilinear map \(\beta\colon \F^k\times \F^k\to Z\), and since every element in \(V^k\) is equivalent to a linear map \(\F^k\to V\), it follows that (\ref{varphi}) is the map above specialized to \(U=\F^k\), \(V=\g/Z(\g)\), \(Z=[\g,\g]\), and \(\omega\) the bilinear map determined by the Lie bracket on \(\g\). 

For an arbitrary \(U\), the fibers of \(\varphi_U\) are then
\[\varphi^{-1}_{U,V}(\beta) = \{f\in\Hom(U,V)\;|\; f^*\omega = \beta\}\,.\]

First we note that the fibers are invariant under automorphisms of \(U\), up to homeomorphism. We record this in the next lemma but the proof is straightforward by definition of \(\varphi_{U,V}\).
\begin{lemma}\label{lem:varphi-equivariant}
Consider the right \(\GL(U)\)-actions over \(\Hom(U,V)\) and \(\Omega^{\mathrm{skew}}(U,Z)\) given by pulling back along automorphisms of \(U\). Then the map \(\varphi_{U,V}\) is \(\GL(U)\)-equivariant. 
In particular, all fibers of \(\varphi_{U,V}\) over a \(\GL(U)\)-orbit are homeomorphic.
\end{lemma}

Let \(\beta\colon U\times U\to Z\) be a bilinear map. Suppose \(W\subset U\) is a complement of \(\ker({\beta}^{\adj})\), and consider \(i^*\beta\oplus 0\colon (W\oplus\ker({\beta}^\adj))\times(W\oplus\ker({\beta}^\adj))\to Z\), where \(i\colon W\to U\) is the inclusion. Then Lemma \ref{lem:varphi-equivariant} yields that up to homeomorphism there is an inclusion given by the composite
 \[\varphi_{W,V}^{-1}(i^*\beta)\hookrightarrow \varphi_{U,V}^{-1}(i^*\beta\oplus 0)\cong \varphi_{U,V}^{-1}(\beta)\] 
where the first map is defined by extending a linear map \(f\colon W\to V\) trivially over \(\ker(\beta^\adj)\). These inclusions are in fact retracts, with retraction induced by restriction along \(W\subset U\).

\begin{lemma}\label{defretX}
Let \(\beta\in \Omega^{\mathrm{skew}}(U,Z)\) and \(i^*\beta\in\Omega^{\mathrm{skew}}(W,Z)\) as above. Then the inclusion \(\varphi_{W,V}^{-1}(i^*\beta)\hookrightarrow  \varphi_{U,V}^{-1}(\beta)\) admits a deformation retraction.
\end{lemma}
\begin{proof}
Let \(f\in \varphi_{U,V}^{-1}(\beta)\). Without loss of generality let us assume that \(i^*\beta\oplus 0 = \beta\). If \(W^\prime\subset U\) is a complement of \(W\), then  \((f^*\omega)|_{W^\prime\times W^\prime}=0\), 
and a linear homotopy would give the desired deformation retraction. Indeed, pick any complement \(W^\prime\) of \(W\) and consider  
\begin{align*}
 H\colon I\times U &\to U\\
(t,(w,w^\prime))&\mapsto (w,tw^\prime)\,.
\end{align*}
Then \(H\) induces a linear homotopy \(\tilde{H}\colon I\times \Hom(U,V)\to \Hom(U,V)\) given by \(\tilde{H}(t,f) = H_t^*(f)\). The above argument shows that \(\tilde{H}\) restricts to a well defined map \(\tilde{H}\colon  I\times \varphi_{U,V}^{-1}(\beta)\to \varphi_{U,V}^{-1}(\beta)\), yielding a deformation retraction, as claimed.
\end{proof}

For an arbitrary skew-symmetric bilinear map \(\beta\) we define the \emph{rank of} \(\beta\) as the codimension of \(\ker(\beta^\adj)\).

\begin{proposition}\label{prop:homotopy type varphi^{-1}}
Every non-empty fiber \(\varphi^{-1}_{U,V} (\beta)\) is homotopy equivalent to a space of linear embeddings of spaces with bilinear maps \[\mathrm{Emb}((W,\alpha),(V,\omega))\,,\] where \(\alpha= i^*\beta\) and \(i\colon W\hookrightarrow U\) is any complement of \(\ker\beta^\adj\). 
\end{proposition}

\begin{proof}
Lemma \ref{defretX} asserts that there is a homotopy equivalence \(\varphi_{W,V}^{-1}(\alpha)\simeq \varphi_{U,V}^{-1}(\beta)\) where \(W\subset U\) is a complement of \(\ker(\beta^\adj)\) and \(\alpha = i^*\beta\). Since \(\alpha\) is non-degenerate it follows that every \(f\in \varphi_{U^\prime,V}^{-1}(\alpha)\) is injective, and hence \(\varphi_{U^\prime,V}^{-1}(\alpha) = \mathrm{Emb}((U^\prime,\alpha),(V,\omega))\).
\end{proof}

For a finite dimensional Lie algebra \(\g\), the canonical isomorphism \(\Omega^{\text{skew}}(\F^k,[\g,\g])\cong\Ant_k([\g,\g])\) induced by the canonical basis of \(\F^k\), allows us to interpret \(\Ant_k(A)\) as a subset of skew symmetric maps with values in \(Z(\g)\cap[\g,\g]\). Understanding \(B_k(\g,A)\), up to homotopy, now reduces to study the homotopy type of spaces of linear embeddings of spaces with bilinear maps 
\[\mathrm{Emb}((\F^r,\alpha),(\g/Z(\g),\omega))\,,\]
where \(r\leq k\), and \(\alpha\colon \F^r\times\F^r\to \mathrm{span}_\g(A)\) is a non-degenerate skew-symmetric bilinear map.

The situation in which \([\g,\g]=Z(\g)\) and \(\mathrm{rank}(Z(\g))= 1\) can be fully described as we will do in the next section, but before that let us record the corresponding statements for nilpotent Lie groups.

\subsection{Commuting elements in nilpotent Lie groups }

Suppose \(G\) is a connected nilpotent Lie group with Lie algebra \(\g\). By Lemma \ref{lem:center-center} the restriction of the exponential map to the center $\exp|\colon Z(\g)\to Z(G)$ is a universal covering of $Z(G)$ and a homomorphism,  hence as abelian Lie groups \(Z(G)\) is isomorphic to the quotient \(Z(\g)/\ker(\exp|)\). Since the inclusion $Z(G)\hookrightarrow G$ admits a deformation retraction, we then have a canonical identification of $\pi_1(G)$ as the central subgroup 
\[\pi_1(G) = \ker(\exp|)\subset Z(\g)\,.\] 
Hereon we will identify $\pi_1(G)$ with this discrete subgroup.


\begin{theorem}\label{maintheorem2}
Let \(G\) be a connected nilpotent Lie group and $k\geq 2$. Then, up to homotopy equivalence, \(\Hom(\Z^k,G)\) is a disjoint union of spaces
\[\mathrm{Emb}((\F^r,\alpha),(\g/Z(\g),\omega))\times (Z(\g)/\pi_1(G))^k\,,\] 
where \(2\leq r\leq k\) and \(\alpha\in H^2(\Z^r;\pi_1(G))\) is non-degenerate, or $r=0$  and $\alpha = 0$. 
\end{theorem}

\begin{remark}
The space corresponding to the case of $r = 0$ and $\alpha = 0$, is path-connected and corresponds to $\Hom(\Z^k,G)_1$, the path-connected component of the trivial representation. Indeed, $\Emb((0,0),(\g/Z(\g),\omega)) = \{0\}$ and the trivial representation is contained in $\{0\}\times (Z(\g)/\pi_1(G))^k$.
\end{remark}

\begin{proof}[Proof of Theorem \ref{maintheorem2}]
From (\ref{hom&ComElLieAlg}) and Lemma \ref{upper-bound-AlmCom} we obtain a homeomorphism \[\Hom(\Z^k,G)\xrightarrow{\cong} \varphi^{-1}(\Ant_{k}(\pi_1(G)))\times(Z(\g)/\pi_1(G))^k\,.\]
As $\pi_1(G)$ is a discrete subgroup of $Z(\g)$, the space $\varphi^{-1}(\Ant_{k}(\pi_1(G)))$ is a disjoint union of fibers $\varphi_{U,V}^{-1}(\beta)$. Proposition \ref{prop:homotopy type varphi^{-1}} specialized to $U = \F^k$, $V = \g/Z(\g)$,  $W = \F^r$, $Z = \mathrm{span}_\g(\pi_1(G))$, and $\omega$ the induced map form the Lie bracket of $\g$, gives the desired deformation retraction.
\end{proof}

A first application of Theorem \ref{maintheorem2} is regarding the question of existence of a deformation retraction of commuting elements in $G$ into commuting elements in a maximal compact subgroup $K$. To this end we will assume that $\pi_0(G)$ is finite, so that the inclusion $K\hookrightarrow G$ admits a deformation retraction. 

To state our result, let us denote the path-connected component of the identity of $G$ by $G_0$.

\begin{corollary}\label{cor:def-ret-1-comp}
Let $G$ be a nilpotent Lie group with finite $\pi_0(G)$, and let $k\geq 2$. If $K\subset G$ is a maximal compact subgroup, then there is a deformation retraction
 \[\Hom(\Z^k,G)_1 \simeq \Hom(\Z^k,K)_1 = K_0^k\,.\]
\end{corollary} 
 
\begin{proof}
In general, for an arbitrary Lie group we always have $\Hom(\Z^k,G)_1 = \Hom(\Z^k,G_0)_1$. Now $\Hom(\Z^k,G_0)_1$ corresponds to the case of $r = 0$ and $\alpha = 0$ of Theorem \ref{maintheorem2} applied to $G_0$. If we let $\g$ denote the Lie algebra of $G_0$ we have a deformation retraction $\Hom(\Z^k,G_0)_1\simeq (Z(\g)/\pi_1(G_0))^k$ with respect to the inclusion $Z(\g)/\pi_1(G_0)\cong Z(G_0)\hookrightarrow G_0$. Now note that $K_0$ is the maximal compact subgroup of $G_0$ (in fact the maximal torus), and thus is the central subgroup where the inclusion $K_0\hookrightarrow Z(G_0)$ admits a deformation retraction.
\end{proof} 
 
The next application concerns the extent of Pettet and Souto's deformation retraction generalization to nilpotent Lie groups.

 \begin{corollary}\label{cor:hom-path-conn}
 Let $G$ be a connected nilpotent Lie group and $k\geq 2$. Then $\Hom(\Z^k,G)$ is not path-connected if and only if there is a non-trivial element in $\pi_1(G)$ that is a Lie bracket. 
\end{corollary}

\begin{remark}
If $\Hom(\Z^k,G)$ is not path connected, then $\Hom(\Z^k,G)\not\simeq \Hom(\Z^k,K)$ because $K$ is a torus and hence $\Hom(\Z^k,K) = K^k$, which is connected. If $\Hom(\Z^k,G)$ is path-connected, then the deformation retraction is attained by Corollary \ref{cor:def-ret-1-comp}.
\end{remark}

\begin{proof}[Proof of Corollary \ref{cor:hom-path-conn}]

First note that Theorem \ref{maintheorem2} implies that $\Hom(\Z^k,G)$ is path-connected if and only if every space of embeddings $\mathrm{Emb}((\F^r,\alpha),(\g/Z(\g),\omega))$ is empty for $r\geq 2$ and $\alpha\in \Ant_r(\pi_1(G))$ non-degenerate.

 Suppose $a\in \pi_1(G)\subset Z(\g)$ is a non-zero bracket, say $a = [x,y]$ for some $x,y\in\g$. Then we define $\alpha\colon \F^2\times\F^2\to  \mathrm{span}_\g(\pi_1(G))$ in the canonical basis $\{e_1,e_2\}$ by $\alpha(e_1,e_2) = [x,y]$ and extend linearly so that $\alpha$ is antisymmetric. Then the map $f\colon \F^2\to \g/Z(\g)$ determined by $f(e_1) = x$ and $f(e_2) = y$, satisfies $f^*\omega = \alpha$. Since $[x,y]$ is not trivial, it follows that $\alpha^\adj$ is injective, and by definition $\alpha$ is non-degenerate. Then $\Emb((\F^2,\alpha),(\g/Z(\g),\omega))$ is non-empty.
 
 Now suppose that there is a linear embedding  $f\colon \F^r\to \g/Z(\g)$ and a non-degenerate bilinear map $\alpha\colon \F^r\times\F^r\to \mathrm{span}_\g(\pi_1(G))$ such that $f^*\omega = \alpha$, for some $r\geq 2$. Since $\alpha$ is non-degenerate, then there exists $1<i\leq r$ such that $\alpha(e_1,e_i) = [f(e_1),f(e_i))]$ is non-zero, completing the proof.
\end{proof}

The main examples as different instances of Corollary \ref{cor:hom-path-conn} are connected nilpotent Lie groups $G$ isomorphic to $(S^1)^m\times N$ where $N$ is simply-connected; in which case all non-trivial elements of $\pi_1(G)$ are \emph{not} Lie brackets, so $\Hom(\Z^k,G)$ is connected. On the other end, we have the reduced Heisenberg groups, which in dimension $2n+1$ may be described as a non-trivial central extension 
\[1\to S^1\to G\to \F^{2n}\to 1\,,\]
and in these cases every element of $\pi_1(G)$ \emph{is} a Lie bracket, and thus $\Hom(\Z^k,G)$ is disconnected. In the next section we will analyze the situation for these latter groups in full detail.

\section{Almost commuting elements in Heisenberg algebras}

%

The 2-step nilpotent Lie algebras with center of rank 1 are the Heisenberg algebras \(\h_{2n+1}(\F)\) and there is exactly one for every \(n\). We have a canonical identification 
\[\h_{2n+1}(\F)/Z(\h_{2n+1}(\F)) = \F^{2n}\]
where the first \(n\) coordinates are identified with \(a+Z(\h_{2n+1}(\F))\) and the last \(n\) coordinates with \(b+Z(\h_{2n+1}(\F))\), where \(a\) and \(b\) are the \(1\times n\) and \(n\times 1\) matrices, respectively, as in \((\ref{strict-upper})\). Formula (\ref{bracket-sut}) specialized to \(\h_{2n+1}(\F)\) reads
 \[[\st_{n+2}(a,b,0,z),\st_{n+2}(a^\prime,b^\prime,0,z^\prime)] = \st_{n+2}(0,0,0,a^\prime b - ab^\prime).\]
That is, the bilinear form induced by the Lie bracket, which we will denote \(\omega_n\), is given by
\begin{align*}
\omega_{n}\colon \F^{2n}\times\F^{2n}&\to \F\\
((x_1,y_1),(x_2,y_2))&\mapsto x_1\cdot y_2 - 
x_2\cdot y_1\,.
\end{align*}
Let \(\{e_j\}_{j=1}^k\) be the canonical basis of \(\F^k\). Then map (\ref{varphi_U}) takes the form
\begin{align*}
\varphi_{k,n}:=\varphi_{\F^k,\F^{2n}}\colon \Hom(\F^k,\F^{2n})&\to \Ant_k(\F)\\
f&\mapsto ((f^*\omega_n)(e_i,e_j))\,.
\end{align*}

\subsection{\(A^k\)-equivariant homeomorphism type}

\def\lf{\left\lfloor}
\def\rf{\right\rfloor}

The following lemma is perhaps the key fact in our analysis of skew-symmetric bilinear maps with values over a field instead of a vector space. 

\begin{lemma}\label{rank-inv}
Let \(k,n\geq1\), and \(B,B^\prime\in \Ant_{k}(\F)\). If \(B\) and \(B^\prime\) have the same rank, then there is a homeomorphism \(\varphi^{-1}_{k,n}(B)\cong \varphi_{k,n}^{-1}(B^\prime)\).
\begin{proof}
Let \(J_m\) be the standard \(2m\times 2m\) symplectic matrix, and suppose \(2r\leq k\). The Darboux basis of a symplectic vector space of rank \(2r\) gives a congruence to the \(k\times k\) block matrix with first \(2r\times 2r\) block equal to \(J_{r}\) and is zero elsewhere. So in particular, if two antisymmetric matrices have the same rank, they are congruent. Hence the map that assigns the rank of a matrix induces a bijection 
\begin{center}
 \(\Ant_{k}(\F)/\GL(k,\F)\xrightarrow{\cong}{\{0,2,4,...,2\lf\frac{k}{2}\rf \}}\).  
\end{center}
where the orbit space is taken with respect to the congruence action of \(\GL(k,\F)\). Since pullback along automorphisms of \(\F^k\) corresponds to congruence of matrices under the canonical isomorphism \(\Omega^{\text{skew}}(\F^k,\F)\cong \Ant_k(\F)\), the lemma now follows from Lemma \ref{lem:varphi-equivariant}.
\end{proof}
\end{lemma}

We now have that the fibers are determined by the rank of the antisymmetric matrix and therefore we may only analyze one for each rank, up to homeomorphism. For convenience let us introduce the following notation. 

\begin{defi}
Let \(k\geq 0\) and \(n\geq 1\). Then for every \(0\leq r\leq \lf \frac{k}{2}\rf\), we define
\[X(k,n,2r):= \{f\in \Hom(\F^k,\F^{2n})\;|\; f^*\omega_n = \omega_{r}\oplus 0_{k-2r}\}\,,\]
which is precisely the preimage of the block matrix \((\omega_{r}(e_i,e_j))\oplus 0_{k-2r}\in \Ant_k(\F)\) under \(\varphi_{k,n}\).
\end{defi}

It is not hard to see that the spaces \(X(k,n,2r)\) are in fact algebraic varieties carved out by a set of quadrics, and we will use this description in Section \ref{examples}. In particular an element in \(X(k,n,2r)\) is a linear transformation that restricts over \(\F^{2r}\) to a symplectic embedding. Up to homotopy, this completely describes \(X(k,n,2r)\). To see this let \(\Sp(\F^{2m},\omega_m)\) be the group of linear automorphisms of \(\F^{2m}\) that preserve the symplectic form \(\omega_m\).

\begin{lemma}\label{some-homeo-typeX}
Let \(k,n\geq1\), and  \( r\leq \lf \frac{k}{2}\rf\). 
\begin{enumerate}
\item If \(n<{r}\), then \(X(k,n,2r) = \emptyset\).
\item If \(n\geq r\), then there is a homotopy equivalence 
\[\Sp(\F^{2n},\omega_n)/\Sp(\F^{2(n-r)},\omega_{n-r})\xrightarrow{\simeq}X(k,n,2r)\]
which is a diffeomorphism if \(k = 2r\).
\end{enumerate}
\end{lemma}

\begin{proof}

1. It follows from the fact that for every \(f\) the matrix \(((f^*\omega_n)(e_i,e_j))\) has at most the rank of \(\mathrm{im} (\omega^{\adj}_n)\) which is \(2n\).

2. Let us first suppose that \(k = 2r\), then \(X(2r,n,2r)\) is the space of symplectic embeddings \((\F^{2r},\omega_r)\hookrightarrow (\F^{2n},\omega_n)\), and thus it admits an action of the Lie group \(\Sp(\F^{2n},\omega_n)\) from the right. The Darboux basis guarantees that the action is transitive, and that the stabilizer of the canonical symplectic embedding is \(\Sp(\F^{2(n-r)},\omega_{n-r})\). Therefore we have a diffeomorphism \(X(2r,n,2r)\cong \Sp(\F^{2n},\omega_n)/\Sp(\F^{2(n-r)},\omega_{n-r})\). If \(k>2r\) we specialize Lemma \ref{defretX} to \(W=\F^r\), \(U = \F^k\) and \(V = \F^{2n}\) to obtain that the inclusion 
\[X(2r,n,2r)\xhookrightarrow{\simeq} X(k,n,2r)\]
admits a deformation retraction, as in this case \(i^*(\omega_r\oplus0_{k-2r}) = \omega_{r}\). 
\end{proof}

\begin{remark}
It is not so clear whether the fact that the space of symplectic embeddings is actually a homogeneous space generalizes to non-degenerate skew-symmetric bilinear forms with values in a vector space of dimension greater than one. In that case there is no Darboux basis nor a complement of a vector subspace in terms of its `symplectic' complement. 
\end{remark}


As a group of matrices, \(\Sp(\F^{2m},\omega_m)\) can be identified with those matrices in \(\GL(2m,\F)\) that preserve the antisymmetric matrix 
\[\Omega_m:= (\omega_m(e_i,e_j))\]
under congruence. Let \(\Sp(2m,\F)\) denote the group of \(2m\times 2m\) matrices that preserve the standard symplectic matrix 
\[ 
J_m = \left(
\begin{matrix}
0&I_m\\
-I_m&0
\end{matrix}\right)
\,.\] We can canonically identify these groups since we can obtain \(\Omega_m\) from \(J_m\) by applying a riffle shuffle permutation.

We can now state our description of almost commuting elements of Heisenberg algebras.

\begin{proposition}\label{prop-com-elem}
Let \(n\geq 1\), and \(A\subset Z(\h_{2n+1}(\F))\) the central subgroup \(\Z\) in the real case and \(\Z^2\) in the complex case. Then the map of ordered commutators in path-connected components
\[\pi_0(B_k(\h_{2n+1}(\F),A))\to H^2(\Z^k;A)\]
is injective for every \(k\geq 2\), and surjective if and only if \(\lf\frac{k}{2}\rf\leq n\). Moreover, each path-connected component parametrized by a form \(\beta\in H^2(\Z^k;A)\) is \(A^k\)-homeomorphic to \(X(k,n,2r)\times\F^k\), where \(2r=\mathrm{rank}(\beta^{\adj})\).
\end{proposition}

\begin{proof}
Let us identify \(H^2(\Z^k,A)\) with \(\Ant_k(A)\). For the injectivity statement, Lemma \ref{upper-bound-AlmCom} tells us we only need to verify that each of the non-empty subspaces 
\[\varphi^{-1}_{k,n}(B)\times \F^{k}
\]
 with \(B\in\Ant_k(A)\), is path-connected. Lemma \ref{rank-inv} asserts that each of these is homeomorphic to \(X(k,n,2r)\times\F^{k}\), where \(2r\) is the rank of the respective \(B\), and Lemma \ref{lem:varphi-equivariant} implies that it is actually an \(A^k\)-equivariant homeomorphism. For every \(n\geq 1\) the symplectic Lie group \(\Sp(2n,\F)\) is connected for both \(\F = \R\) and \(\C\), hence by Lemma \ref{some-homeo-typeX} part 2 it follows that each \(X(k,n,2r)\) is path-connected, as well. The surjectivity statement follows from Lemma \ref{some-homeo-typeX} part 1.
\end{proof}

\subsection{Commuting elements in reduced generalized Heisenberg groups}

Let \(\O(m)\) be the group of orthogonal \(m\times m\) real matrices, \(\U(m)\) the group of \(m\times m\) unitary matrices, and let \(\Sp(m)\) be the \(2m\times 2m\) compact symplectic group (which may be interpreted as the orthogonal \(n\)-frames in the quaternionic space, \(\mathbb{H}^n\)). The maximal compact subgroup of \(\Sp(2m,\R)\) is \(\O(2m)\cap \Sp(2m,\R)\) which can be seen is isomorphic to \(\U(m)\). Similarly, the maximal compact subgroup of \(\Sp(2m,\C)\) is \(\U(2m)\cap \Sp(2m,\C)=\Sp(m)\). For simplicity, let us denote by \(\Or\Sp(2n,\F)\subset\Sp(2n,\F)\) the maximal compact subgroup described above. 

 As mentioned before, the space \(X(2r,n,2r)\) is the Stiefel manifold of \(2r\)-symplectic frames in \(\F^{2n}\). To distinguish between \(\F = \R\) or \(\C\), let us denote this space as \(\Sp(2n,2r,\F)\), as well. Similarly, let \(\Or\Sp(2r,2n,\F)\) denote the subspace of orthonormal \(2r\)-symplectic frames in \(\F^{2n}\). Again, as mentioned in the proof of Lemma \ref{some-homeo-typeX} the Darboux basis implies that the groups \(\Or\Sp(2n,\F)\) and \(\Sp(2n,\F)\) act transitively over \(\Or\Sp(2n,2r,\F)\)  and \(\Sp(2n,2r,\F)\), respectively. Hence we have diffeomorphisms
\begin{align}
\Sp(2n,2r,\F)&\cong\Sp(2n,\F)/\Sp(2n-2r,\F)\label{symp-complex-Stiefel}
\end{align}
\begin{align}
\Or\Sp(2n,2r,\F)&\cong\Or\Sp(2n,\F)/\Or\Sp(2n-2r,\F)\,.\label{or-symp-complex-Stiefel}
\end{align}

\begin{remark}\label{rem-orframes=orsympframes}
Recall that \(D(\F)\) denotes the Cayley--Dickson doubling of \(\F\), so that \(D(\R)=\C\) and \(D(\C)=\mathbb{H}\). We emphasize that (\ref{or-symp-complex-Stiefel}) and the identification of the maximal compact subgroups of the symplectic groups yield a diffeomorphism 
\[\Or\Sp(2n,2r,\F)\cong V_r(D(\F)^n)\]
between the Stiefel manifold of orthogonal symplectic \(2r\)-frames in \(\F^{2n}\) and the Stiefel manifold of orthogonal \(r\)-frames in \(D(\F)^n\). The key observation for example when \(\F=\R\) is that the hermitian product carries two forms in its real and imaginary parts: the real inner product and the standard symplectic form.
\end{remark}

%

\begin{lemma}\label{lem:symp-complex-Stiefel}
There is a homotopy equivalence
\[\Or\Sp(2n,2r,\F)\simeq \Sp(2n,2r,\F)\]
\end{lemma}
\begin{proof}
The inclusions \(\Or\Sp(2m,\F)\hookrightarrow\Sp(2m,\F)\), and diffeomorphisms (\ref{symp-complex-Stiefel}) and (\ref{or-symp-complex-Stiefel}) induce a map of fibrations \[\Or\Sp(2n,2r,\F)\to\Sp(2n,2r,\F).\] The long exact sequence of homotopy groups applied to this map of fibrations and the Five lemma, together show that the induced map is a weak homotopy equivalence, and hence a homotopy equivalence.
\end{proof}

\begin{theorem}\label{main-theorem}
Let \(n\geq 1\), and let \(\widetilde\He_{2n+1}(\F) = \He_{2n+1}(\F)/A\) be the reduced generalized Heisenberg group of dimension \(2n+1\), where \(A= \Z\) in the real case and \(A=\Z^2\) in the complex case. Then the map of ordered commutators in connected components
\[\pi_0(\Hom(\Z^k,\widetilde\He_{2n+1}(\F)))\to H^2(\Z^k;A)\]
is injective for every \(k\geq 2\), and surjective if and only if \(\lf\frac{k}{2}\rf\leq n\). Moreover, each component parametrized by a form \(\beta\in H^2(\Z^k;A)\) is homotopy equivalent to \[V_r(D(\F)^n)\times (S^1)^{k\dim_\R\F }\,,\] where \(2r=\mathrm{rank}(\beta^{\adj})\).
\end{theorem}

\begin{proof}
The injectivity and surjectivity statements follow from Proposition \ref{prop-com-elem}. Moreover, it implies that each path-connected component is homeomorphic to a respective \(X(k,n,2r)\times(\F/A)^{k}\). By Lemma \ref{some-homeo-typeX} and Lemma \ref{symp-complex-Stiefel} it follows that each \(X(k,n,2r)\times(\F/A)^{k}\) is homotopy equivalent to \((U(n)/U(n-r))\times (S^{1})^{k}\) in the real case and to \((\Sp(n)/\Sp(n-r))\times (S^{1})^{2k}\) in the complex case, as claimed.
\end{proof}

\begin{proof}[Proof of Theorem \ref{main-theoremST}]
The statement for the reduced groups of upper unitriangular matrices easily follows from Corollary \ref{cor: hom-equiv-comm-elem-Hien-Upper} and Theorem \ref{main-theorem}.
\end{proof} 

%

Recall that the moduli space of isomorphism classes of flat connection, or the representation space $\Rep(\Z^k,G)$ is the orbit space of $\Hom(\Z^k,G)$ under the conjugation action of $G$, and that the map that forgets the flat connection descends to the quotient, yielding  $\Rep(\Z^k,G)\to H^2(\Z^k,\pi_1(G))$.

\begin{corollary}\label{cor:rep-redHeis}
Let \(n\geq 1\). Then map that forgets the flat connection in path-connected components of representation spaces
\[\pi_0(\Rep(\Z^k,\widetilde\He_{2n+1}(\F)))\to H^2(\Z^k;A)\]
is injective for every \(k\geq 2\), and surjective if and only if \(\lf\frac{k}{2}\rf\leq n\). 
\end{corollary}
\begin{proof}
It follows from a standard argument regarding injective composites. Indeed, the composite  $\Hom(\Z^k,\widetilde\He_{2n+1}(\F))\to \Rep(\Z^k,\widetilde\He_{2n+1}(\F))\to H^2(\Z^k;A)$ is injective in $\pi_0$ by Theorem \ref{main-theorem}, and hence the quotient map is a bijection in $\pi_0$, which in turn makes $\pi_0(\Rep(\Z^k,\widetilde\He_{2n+1}(\F)))\to H^2(\Z^k;A)$ an injection, as desired.
\end{proof}

\subsection{Application to classifying spaces for commutativity}\label{sec:looped-comm-map}

Recall that Heisenberg groups arise from symplectic vector spaces, for instance \(\He_{2n+1}(\R)\) arises from the antisymmetrization of the dot product in \(\R^n\), which we have been denoting by \(\omega_n\). Similarly, we can consider \(\He_{\infty}(\R)\), the Heisenberg group associated to the antisymmetrization of the dot product in \(\R^\infty\). Then, the inclusions \(\R^n=0\oplus\R^n\subset\R^{n+1}\) induce an isomorphism \(\He_{\infty}(\R)\cong\colim_{n}\He_{2n+1}(\R)\), which is equivariant with respect to the central \(\Z\)-action, and  hence it yields a homeomorphism
\[\Hom(\Z^k,\widetilde{\He}_{\infty}(\R))\cong\Hom(\Z^k,\colim_{n}\widetilde{\He}_{2n+1}(\R))\cong\colim_n\Hom(\Z^k,\widetilde{\He}_{2n+1}(\R))\,.\]
Then, by Theorem \ref{main-theorem} the path-connected components of \(\Hom(\Z^k,\widetilde{\He}_{\infty}(\R))\) are in bijection with $H^2(\Z^k;\Z) = \Ant_k(\Z)$ and are homotopy equivalent to \(V_r(\C^\infty)\times (S^1)^k\) for some \(r\). Since $V_r(\C^\infty)$ is contractible, the map that forgets the flat structure 
\[\Hom(\Z^k,\widetilde{\He}_{\infty}(\R))\xrightarrow{\simeq} \Ant_k(\Z)\times(S^1)^k\]
is a homotopy equivalence. Moreover, this equivalence is compatible with the following simplicial structure. For any topological group \(G\), the assignment \(k\mapsto \Hom(\Z^k,G)\) yields a simplicial space whose simplicial structure is obtained by the level-wise restriction of the nerve of \(G\), \(k\mapsto G^k\). We denote its geometric realization by \(B(2,G):=|k\mapsto \Hom(\Z^k,G)|\) (see \cite{AdemCohenTorres} and \cite{AdemGomez} for the basic theory of $B(2,G)$). Recall that \(\Hom(\Z^k,\widetilde{\He}_{2n+1}(\R))\) splits off \((S^1)^k\) for every \(n\geq 1\), then it is not hard to see that the forgetful map carries the simplicial structure described above, to the diagonal simplicial space \(k\mapsto\Ant_k(\Z)\times (S^1)^k\), with respect to the bisimplicial space whose second factor is the nerve of \(S^1\), and on the first factor the \(i\)-th face map simultaneously adds the \(i\)-th column to the \(i+1\) column and the \(i\)-th row to the \(i+1\) row. The geometric realization of \(k\mapsto \Ant_k(\Z)\) is a \(K(\Z,2)\), as it is the \(\Z\)-linearization of the 2-simplex modulo its boundary. Then upon taking realizations we obtain a homotopy equivalence 
\[B(2,\widetilde{\He}_{\infty}(\R))\xrightarrow{\simeq}BS^1\times BS^1\,,\]
in which the inclusion of the classifying map of the center \(BS^1\to B(2,\widetilde{\He}_{\infty}(\R))\) yields a splitting. 

For an arbitrary (well based) topological group $G$, if we consider the simplicial model of the classifying space \(BG\) given by the nerve of \(G\), the level-wise inclusions \(\Hom(\Z^k,G)\subset G^k\) induce a map \(\iota \colon B(2,G)\to BG\). The pullback of the universal $G$-bundle $EG\to BG$ along $\iota$ is is denoted \(E(2,G)\), which is also a model for the homotopy fiber of $\iota$. To understand what the map \(\iota\) looks like for \(G=\widetilde{\He}_{\infty}(\R)\), note that the inclusion of the center \(S^1\hookrightarrow\widetilde{\He}_{\infty}(\R)\) is a homotopy equivalence and hence \(BS^1\to B\widetilde{\He}_{\infty}(\R)\) is a homotopy equivalence, as well. Then, we can see that \(\iota\colon B(2,\widetilde{\He}_{\infty}(\R))\to B\widetilde{\He}_{\infty}(\R)\), up to homotopy, is the projection \(\pi_2\colon BS^1\times BS^1\to BS^1\) into the second factor, and we may further conclude that
\[E(2,\widetilde{\He}_{\infty}(\R))\simeq BS^1\,,\]
where the canonical map \(E(2,\widetilde{\He}_{\infty}(\R))\to B(2,\widetilde{\He}_{\infty}(\R))\) is homotopic to the inclusion into the first factor \(i_1\colon BS^1\to BS^1\times BS^1\). In \cite{A-CGV21} we defined a map \(\mathfrak{c}\colon E(2,G)\to B[G,G]\) arising from the commutator operation in \(G\), and it is an interesting question whether or not the looped commutator map, \(\Omega\mathfrak{c}\), splits up to homotopy, which in particular would imply that the commutator subgroup \([G,G]\) splits off \(\Omega E(2,G)\). In our example we already have the latter statement, that is \[\Omega E(2,\widetilde{\He}_{\infty}(\R))\simeq S^1=[\widetilde{\He}_{\infty}(\R),\widetilde{\He}_{\infty}(\R)]\,,\]
but we claim that the equivalence is of loop spaces via the looped commutator map. According to \cite[Remark 11]{A-CGV21} we may compute \(\mathfrak{c}\) for \(G=\widetilde{\He}_{\infty}(\R)\), up to homotopy, as the composite 
\[BS^1\xrightarrow{i_1} BS^1\times BS^1\xrightarrow{\phi}BS^1\times BS^1\xrightarrow{\pi_2} BS^1\]
where \(\phi\) is defined as follows: The group inverse \((-)^{-1}\colon G\to G\) induces a simplicial map \(B_\bullet(2,G)\to B_\bullet(2,G)\) and \(\phi^{-1}\) is the geometric realization. Then \(\phi\) is the bottom arrow of the square
\[
\xymatrix{
B(2,\widetilde{\He}_{\infty}(\R))\ar[d]_{\simeq}\ar[r]^{\phi^{-1}}&B(2,\widetilde{\He}_{\infty}(\R))\ar[d]_{\simeq}\\
BS^1\times BS^1\ar[r]_\phi&BS^1\times BS^1
}
\]
Since the inverse map for matrices in a generalized Heisenberg group is given by \(Id + \h(a,b,c)\mapsto Id + \h(-a,-b,ab - c)\), we may see that \(\phi\) maps a pair to its `quotient' into the second factor. In particular the commutator map up to homotopy is the identity \(BS^1\to BS^1\), and hence \(\Omega\mathfrak{c}\colon \Omega E(2,\widetilde{\He}_{\infty}(\R))\to \Omega BS^1\) is a homotopy equivalence of loop spaces. \\

There are very few examples known of connected Lie groups whose looped commutator map splits, the list so far being $SU(2)$, $U(2)$, and the stable special unitary group $SU$.

\subsection{Some explicit homeomorphism types}\label{examples}

For every generalized Heisenberg group we will exhibit the homeomorphism type of each component of the space of commuting pairs \(\Hom(\Z^2,\widetilde\He_{2n+1}(\R))\). For arbitrary \(k\), we will address the question of finding the parametrizing set of the components for the case of the reduced Heisenberg group, \(\widetilde{\He}_3(\R)\).

 It will be convenient to have in mind the interpretation of antisymmetric matrices in \(\F\) as the exterior square of \(\F^k\). It can be given via the isomorphism \(\Lambda^2\F^k {\to}\Ant_{k}(\F)\),  \(x\wedge y\mapsto xy^T-yx^T\), where we consider \(x,y\) as column vectors, and \(x^T,y^T\) denote their transpose matrices. We will identify a wedge of vectors in \(\F^k\), \(x\wedge y\), with its corresponding antisymmetric matrix \(xy^T-yx^T\) in \(\Ant_{k}(\F)\).

Let \(\{e_j\}_{j=1}^k\) be the canonical basis of \(\F^k\). For every \(0\leq r\leq \lf\frac{k}{2}\rf\) the antisymmetric matrix \(\Omega_{r}\) defined before has a representative
 \[\Omega_{r}\oplus 0_{k-2r}=e_1\wedge e_2+e_3\wedge e_4+\cdots+e_{2r-1}\wedge e_{2r}\,,\] 
 where here \(\oplus\) means block sum of matrices, and we let \(\Omega_0 = 0\).

Identifying \(\Hom(\F^k,\F^{2n})\) with \(\Mat_{2n\times k}(\F)\), one can easily verify that map \((\ref{varphi_U})\) takes the form
\begin{align*}
\varphi_{k,n}:=\colon \Mat_{2n\times k}(\F)&\to \Ant_{k}(\F)\\
M&\mapsto M^T\Omega_nM\,.
\end{align*}
In this notation we have that \(X(k,n,2r) = \{M\in \Mat_{2n\times k}(\F)\,|\,M^T\Omega_nM = \Omega_r\oplus 0_{k-2r}\}\).

\begin{remark}\label{simplectic-emb-form}
Let \(M\in \Mat_{2n\times k}(\F)\). For every \(1\leq i\leq n\), let \(x^{(i)}\) and \(y^{(i)}\) denote the \(2i-1\) and \(2i\) row
 vectors of \(M\), respectively. We define
\begin{align*}
\tilde\varphi_{k,n}\colon \Mat_{2n\times k}(\F)&\to \Ant_{k}(\F)\\
M&\mapsto \sum_{i=1}^n x^{(i)}\wedge y^{(i)}\,.
\end{align*}
We claim that \(\tilde{\varphi}_{k,n}=\varphi_{k,n}\). Indeed, we have
\[M^T\Omega_nM = \sum_{i=1}^nM^Te_{2i-1}e_{2i}^TM - M^Te_{2i}e_{2i-1}^TM = \sum_{i=1}^n x^{(i)}(y^{(i)})^T-y^{(i)}(x^{(i)})^T\,,\]
which is precisely the definition of \(\varphi_{k,n}(M)\).
\end{remark}

\subsubsection{Commuting pairs in reduced generalized Heisenberg groups}

We will focus in the case \(\F = \R\). Let \(C(X)\) denote the open cone over a space \(X\), namely \(C(X) := \left( [0,1)\times X \right)/\left( \{0\}\times X \right)\).

\begin{lemma}\label{some-homeo-typeX2}
Let \(k,n\geq1\). Then, we have the following homeomorphisms:
\begin{enumerate}
\item \(X(2,n,2)\cong S^{2n-1}\times \R^{2n}\).
\item \(X(2,n,0)\cong C(S^{2n-1}\times S^{2n-1})\).

\end{enumerate}
\end{lemma}
\begin{proof}

1. This case reduces to solve the equation \(x^{(1)}\wedge y^{(1)}+\cdots +x^{(n)}\wedge y^{(n)}  = e_1\wedge e_2\), or equivalently 
\[\sum_{i=1}^n\left|\begin{matrix}
x_{1}^{(i)}&x_{2}^{(i)}\\
y_{1}^{(i)}&y_{2}^{(i)}
\end{matrix}
\right| = 1\,.\]
For \(j=1,2\), let us consider \(x_j\) and \(y_j\) as the \(n\times 1\) column vectors whose \(i\)-th entry is \(x_{j}^{(i)}\) and \(y_{j}^{(i)}\), respectively. Then the above equation is \(x_1^Ty_2-x_2^Ty_1 = \Omega_2\). Let \(u_1,u_2,v_1,v_2\in \R^n\) and write \(x_1 = u_1-v_1\), \(y_2 = u_1 + v_1\),  \(x_2 = v_2 - u_2\), and \(y_1 = v_2 + u_2\). The equation now reads as
\[1+ \|v_1\|^2 + \|v_2\|^2 = \|u_1\|^2  + \|u_2\|^2\,. \]
We can express the solution set as the vectors \((u_1,u_2)\in \R^{2n}\) of a fixed length \(\sqrt{1+ \|v_1\|^2 + \|v_2\|^2}\), and we have one of such spheres for every \((v_1,v_2)\in\R^{2n}\). Explicitly the map \(X(2,n,2)\to S^{2n-1}\times\R^{2n}\) that sends \((x_1,x_2,y_1,y_2)\) to
\[\left(\frac{x_1+y_2}{\sqrt{4+ \|y_1+x_2\|^2 + \|y_2-x_1\|^2}},\frac{y_1-x_2}{\sqrt{4+ \|y_1+x_2\|^2 + \|y_2-x_1\|^2}},\frac{y_2-x_1}{2},\frac{y_1+x_2}{2}\right)\]
is well defined, and is clearly a homeomorphism.

2. The equation defining \(X(2,n,0)\) is \(x^{(1)}\wedge y^{(1)}+\cdots +x^{(n)}\wedge y^{(n)}  = 0\), and the same change of variables as in the proof of part 1 yields the identity \(\|v_1\|^2 + \|v_2\|^2 = \|u_1\|^2  + \|u_2\|^2\), and as long as \(v_1\) or \(v_2\ne 0\) a similar reasoning as above gives a homeomorphism \[X(2,n,0)-\{0\}\cong S^{2n-1}\times(\R^{2n}-\{0\})\,.\]
We can see that the radius of the spheres decreases as the \(\R^{2n}\)-coordinate tends to zero, and since the above homeomorphism in turn induces a homeomorphism \(X(2,n,0)-\{0\}\cong S^{2n-1}\times S^{2n-1}\times (0,1)\), we may conclude that after adjoining the zero map we obtain the cone over \(S^{2n-1}\times S^{2n-1}\) as claimed.
\end{proof}


\begin{proposition}\label{alcom-pairs}
For every \(n\geq 1\) there is a homeomorphism
\[\Hom(\Z^2, \widetilde{\He}_{2n+1}(\R))\cong \left(C(S^{2n-1}\times S^{2n-1})\sqcup\bigsqcup_{\Z-\{0\}}S^{2n-1}\times\R^{2n}\right)\times(S^1)^2  \,,\] 
where the component \(C(S^{2n-1}\times S^{2n-1})\times(S^1)^2\) is the path-connected component of the trivial representation. 
\end{proposition}

\begin{proof}
First, note that the map \(\varphi_n\colon \Mat_{2n\times 2}(\R)\to \Aut_{2\times 2}(\R)\cong \R\) is surjective, since this is a non-zero bilinear map. By Lemma \ref{upper-bound-AlmCom}, Lemma \ref{some-homeo-typeX} part 2, and Lemma \ref{some-homeo-typeX2} part 1, for every \(n\geq 1\) there is a homeomorphism
\[B_2(\He_{2n+1}(\R),\Z)\cong C(S^{2n-1}\times S^{2n-1})\times \R^2\sqcup\left(\bigsqcup_{\Z-\{0\}}S^{2n-1}\times\R^{2n}\right)\times\R^2  \,,\] 
where the component \(C(S^{2n-1}\times S^{2n-1})\times\R^2\) corresponds to \(\Hom(\Z^2,\He_{2n+1}(\R))\). The right hand side union corresponds to the rank 2 matrices. The result follows after taking the \(\Z^2\)-orbit space as in (\ref{hom&ComElLieAlg}).
\end{proof}

\subsubsection{Commuting elements in the reduced Heisenberg group}

Let us briefly recall that every element in \(B_k(\h_{3}(\R),\Z)\) is a \(2\times k\) matrix \(M\) such that the every minor of \(M\) is an integer. This data is encoded in the vector \(\varphi_{k,1}(M)\) belonging to the subset \(\Ant_k(\Z)\) of the exterior algebra \(\bigwedge^{2}\R^k\cong\Ant_k(\R)\) whose coordinates are the ordered \(2\times 2\) minors of \(M\). The correspondence \(M\mapsto \varphi_{k,1}(M)\) is injective, up to change of basis. Indeed, two \(2\times k\) matrices, \(M\) and \(M^\prime\), with equal minors span the same 2-plane, hence there is a matrix \(A\in \SL(2,\R)\) such that \(M=AM^\prime\). In other words, upon passing to \(\SL(2,\R)\)-orbits, \(\varphi_{k,1}\) induces an embedding 
\begin{align}
(\Mat_{2\times k}(\R)-X(k,1,0))/\SL(2,\R)\hookrightarrow \bigwedge^{2}\R^k-\{0\}\,.\label{varphi-k-1}
\end{align}
Perhaps a more familiar interpretation of the quotient space \((\Mat_{2\times k}(\R)-X(k,1,0))/\SL(2,\R)\) is as the punctured affine cone of the Grassmannian of 2-planes, \(\mathrm{Gr}(2,\R^k)\). In particular the Pl\"ucker relations show that the map (\ref{varphi-k-1}) is not surjective.

 It follows from (\ref{varphi-k-1}) that \(X(k,1,2)\) is homeomorphic to \(S^1\times\R^2\). In the last section we will show that \(X(k,1,0)-\{0\}\cong (S^{k-1}\times_{\Z/2} S^1)\times \R\). In the mean time we use this to prove:

\begin{proposition}\label{prop:comm-elements-red-heisen}
For every \(k\geq 3\), the non-identity path-connected components of \(\Hom(\Z^k,\widetilde\He_3(\R))\) are in 1-1 correspondence with the set \(\mathcal{M}\) of integral Pl\"ucker coordinates of the affine cone of the Grassmannian of 2-planes \(C(\mathrm{Gr}(2,\R^k))\). Moreover, there is a homeomorphism
\[\Hom(\Z^k,\widetilde\He_3(\R))\cong \left(C(S^{k-1}\times_{\Z/2} S^1)\sqcup\bigsqcup_{\mathcal{M}}S^1\times\R^2\right)\times(S^1)^k\,.\]
\end{proposition}

\section{Representation spaces in reduced generalized Heisenberg groups}

In light of Remark \ref{rem:G-action-on-ComLieAlg}, it will be easier to describe first the quotient space \(B_k(\h_{2n+1}(\F),A)\) under the conjugation action of \(\He_{2n+1}(\F)\) in order to better understand the topology of \(\Rep(\Z^k,\widetilde\He_{2n+1}(\F))\).

\subsection{The \(\He_{2n+1}(\F)\) conjugation action on \(B_k(\h_{2n+1}(\F),A)\)}

 For simplicity, we will denote the elements of the Heisenberg algebras as
\[\h_{2n+1}(a,b,z) := \st_{n+2}(a,b,0,z) \,,\]
where \(a\in \Mat_{1\times n}(\F)\), \(b\in \Mat_{n\times 1}(\F)\) and \(z\in \F\). Let us also identify an abelian subgroup \(A\subset \F\) with the central subgroup in \(\h_{2n+1}(\F)\) given by
\[ \{\h_{2n+1}(0,0,z): z\in A\}\,.\]

 Let us first spell out the action over a single \(\h_{2n+1}(\F)\), or in other words, what the adjoint representation of \(\He_{2n+1}(\F)\) looks like. The elements in the orbit of  \(\h_{2n+1}(a,b,z)\in\h_{2n+1}(\F)\) are of the form
 \begin{align*}
\exp(\h_{2n+1}(x,y,w))\cdot\h_{2n+1}(a,b,z) &= \h_{2n+1}(a,b,z) + \h_{2n+1}(0,0,ay-xb) \\
&= \h_{2n+1}(a,b, z+\omega_n((a,b),(x,y)))
\end{align*}
for some \(x\in\Mat_{1\times n}(\F)\), \(y\in\Mat_{n\times 1}(\F)\) and \(z\in\F\), where we identify \((x,y)\) with a vector in \(\F^{2n}\) in the obvious way. The orbit of an arbitrary tuple \((f,z)\in \Hom(\F^k,\F^{2n})\times\F^k\) is then 
\begin{align}
(\exp(\h_{2n+1}(x,y,w)),(f,z))\mapsto (f, (f^*(\omega_n^{\adj}(x,y))(e_1),...,f^*(\omega_n^{\adj}(x,y))(e_k)) + z )\,.\label{conj-action}
\end{align}
 Note that since \(\omega_n^{\adj}\) is an isomorphism, the image of \((x,y)\mapsto f^*(\omega_n^{\adj}(x,y))\) is the same as the image of \(f^*\), hence we can conclude that the \(\He_{2n+1}(\F)\)-orbit under linear conjugation of the \(k\)-tuple represented by \((f,z)\) is homeomorphic to \(\{f\}\times\mathrm{coker}(D\circ f^*)\), where \(D\) is the isomorphism \(D\colon (\F^k)^*\to \F^k\) induced by the canonical basis. Moreover, the quotient map onto the space of orbits is given by 
\begin{align}\label{orbitmap}
B_k(\h_{2n+1}(\F),A)&\to B_k(\h_{2n+1}(\F),A)/\He_{2n+1}(\F)\\
(f,z)&\mapsto (f,z+\mathrm{im}(D\circ f^*))\,.\nonumber
\end{align}
We can give a simpler description of map (\ref{orbitmap}) and the orbit space. Recall that Proposition \ref{prop-com-elem}  every component of \(B_k(\h_{2n+1}(\F),A)\) is homeomorphic to \(X(k,n,2r)\times \F^k\), for some \(r\leq \lf \frac{k}{2}\rf\). We will now describe a subspace that completely represents the orbit space, up to homeomorphism. Let
 \[E(k,n,2r):=\{(f,v)\in X(k,n,2r)\times\F^k:v\in\ker f\}\,.\]

\begin{lemma}\label{conj-identif}
The inclusion \(E(k,n,2r)\hookrightarrow X(k,n,2r)\times\F^k\) induces a homeomorphism 
\[E(k,n,2r)\xrightarrow{\cong}(X(k,n,2r)\times\F^k)/\He_{2n+1}(\F)\,,\]
whose inverse is induced by the retraction \(X(k,n,2r)\times \F^k\to E(k,n,2r)\) given by \((f,z)\mapsto (f,\rho_f(z))\), where \(\rho_f\colon \F^k\to \ker f\) is the orthogonal projection 
\end{lemma}
\begin{proof}
Recall that the composite \(\ker f\hookrightarrow\F^k\to\mathrm{coker}(D\circ f^*)\) is an isomorphism. Therefore the quotient map \(E(k,n,2r)\to(X(k,n,2r)\times\F^k)/\He_{2n+1}(\F)\) given by \((f,v)\to (f,v+\mathrm{im}(D\circ f^*))\) is a continuous bijection with inverse \((f,z+\mathrm{im}(D\circ f^*))\mapsto (f,\rho_f(z))\).
\end{proof}

\begin{remark}\label{remark not-equivariant}
If \(A\ne\{0\}\), then \(E(k,n,2r)\) is not \(A^k\)-invariant with respect to the central action, unless \(k=2r\). In particular the homeomorphism of Lemma \ref{conj-identif} is not \(A^k\)-equivariant. Instead, the \(A^k\)-action that  \(E(k,n,2r)\) inherits from the RHS orbit space is given by \[y\cdot(f,v)=(f,v+\rho_f(y))\,.\]
We will refer to this action as the reduced central action.
\end{remark}

\begin{proposition}
Every path-connected component of \(B_k(\h_{2n+1}(\F),A)/\He_{2n+1}(\F)\) is \(A^k\)-homeomorphic to some \(E(k,n,2r)\) with respect to the reduced central action of \(A^k\). Moreover, the quotient map 
\[B_k(\h_{2n+1}(\F),A)\xrightarrow{\simeq} B_k(\h_{2n+1}(\F),A)/\He_{2n+1}(\F)\]
is a homotopy equivalence.
\end{proposition}
\begin{proof}
Let \(\beta\in H^2(\Z^k;A)\). The path-connected component of \(B_k(\h_{2n+1}(\F),A)/\He_{2n+1}(\F)\) representing \(\beta\) is either empty or has the homotopy type of \(X(k,n,2r)\) where \(r\) is the rank of \(\beta\). Indeed, let \(\tilde{\varphi}_{n,k}\colon B_k(\h_{2n+1}(\F),A)/\He_{2n+1}(\F)\to \Ant_k(A)\) be the induced map of \(\varphi_{k,n}\) in the quotient. Suppose that \(\tilde{\varphi}_{k,n}^{-1}(\beta)\) is non empty. Then by Lemma \ref{conj-identif}, \(\tilde{\varphi}_{k,n}^{-1}(\beta)\) is homeomorphic to \(E(k,n,2r)\). The linear homotopy \(((f,v),t)\mapsto (f,tv)\) gives a deformation of \(E(k,n,2r)\) into \(X(k,n,2r)\), as claimed.
\end{proof}

\subsection{The reduced \(A^k\)-action}

To obtain the representation space \(\Rep(\Z^k,\widetilde\He_{2n+1}(\F))\) from \(B_k(\h_{2n+1}(\F),A)/\He_{2n+1}(\F)\) we need to take a further quotient by the \(A^k\)-action (see Remark \ref{rem:G-action-on-ComLieAlg}). It is clear that it commutes with the conjugation action (\ref{conj-action}). Then we can write the quotient map after both actions as

\begin{align*}
B_k(\h_{2n+1}(\F),A)&\to\Rep(\Z^k,\widetilde\He_{2n+1}(\F))\\
(f,z)&\mapsto (f,z+\mathrm{im}(D\circ f^*)+A^k)=(f,\rho_f(z)+\rho_f(A^k))\,.
\end{align*}
as noted in Remark \ref{remark not-equivariant}, where \(\rho_f\colon \F^k\to \ker f\) is the orthogonal projection. 

\begin{remark}In particular, if we represent an element in \(\Hom(\Z^k, \widetilde\He_{2n+1}(\F))\) as a pair \((f,z+A^k)\), where \(f\colon \F^k\to \F^{2n}\) and \(z\in\F^k\), then the orbit map with respect to the \(\widetilde\He_{2n+1}(\F)\) conjugation action over \(\Hom(\Z^k,\widetilde\He_{2n+1}(\F))\), is given by
\begin{align}
\Hom(\Z^k, \widetilde\He_{2n+1}(\F))&\to \Rep(\Z^k, \widetilde\He_{2n+1}(\F))\nonumber\\
(f,z+A^k)&\mapsto(f,\rho_f(z)+\rho_f(A^k))\,.\nonumber
\end{align}
\end{remark}

Let us now give an example of the reduced \(A^k\)-action on the component of the trivial representation of the smallest orbit space, \(B_2(\h_3(\R),\Z)/\He_3(\R)\).

\begin{example}\label{example-irrational-tori}
Consider \(E(2,1,0)\) with \(\F=\R\). We can represent these elements by pairs \((M,v)\), where \(M\) is a real \(2\times 2\) zero determinant matrix, \(v\in\R^2\) and \(Mv=0\). Let \(\mu\) be an irrational number and define 
\[M = \left(
\begin{matrix}
\mu&1\\
\mu&1
\end{matrix}\right)\,.\]
Then \(v\in \ker M\) if and only if it lies on the line with slope \(-\mu\). In particular, the orthogonal projection \(\rho_M\colon \R^2\to \ker M \) maps \(\Z^2\) into a rank 2 abelian subgroup, that is, \(\rho_M(\Z^2)\subset \ker M \) is a dense subspace. Thus, the orbit of an element \((M,v)\) is homeomorphic to \(\R/(\Z+\mu\Z)\) an ``irrational torus" which as a topological space is an indiscrete space.
\end{example}

In light of the pathological behavior of the \(A^k\)-action over \(E(k,n,2r)\), a common alternative is to consider the \(\mathcal{T}_1\)-quotient or the maximal Hausdorff quotient of \(\Rep(\Z^k,\widetilde\He_{2n+1}(\F))\), but before doing that let us dive a little further into this pathology. Let us consider the honest orbit space
\[\widetilde{E}(k,n,2r):= E(k,n,2r)/A^k\,.\]
To see what the points over \(\widetilde{E}(k,n,2r)\) look like, note that the projection into the first coordinate \(p\colon E(k,n,2r)\to X(k,n,2r)\) is \(A^k\)-equivariant with respect to the trivial action on \(X(k,n,2r)\), hence it induces a map \[q\colon \widetilde{E}(k,n,2r)\to X(k,n,2r)\,,\]
 whose fibers are \(q^{-1}(f) = \ker f/\rho_f(A^k)\). To decide whether these fibers are actual tori or irrational tori, consider the rational points of the variety \(X(k,n,2r)\), that is, the subspace 
\[X(k,n,2r)_\Q = \{f\in X(k,n,2r)\,|\,\text{rank}(\rho_f(A^k)) = \dim(\ker f) \}\,.\]
where we consider the dimension of \(\ker f\subset \F^k\) as a real vector space.

\begin{remark}
  Perhaps it is more common to say that a point \(V\) in the Grassmannian \(\Gr(d,\F^k)\) is rational if \(V\cap A^k\) has rank \(d\), but this is equivalent to \(\rho_V(A^k)\) having rank \(d\), as well, where \(\rho_V\colon \F^k\to V\) is the orthogonal projection. Indeed, if \(V\cap A^k\) has basis \(\{v_1, \ldots, v_d\}\) then for any point \(a \in A^k\), the projection \(\rho_V(a)\) is \(\sum_{i=1}^d \lambda_iv_i\) where \((\lambda_1, \ldots, \lambda_d)\) is the solution of the system of equations \(v_j \cdot (a-\sum_{i=1}^d \lambda_iv_i) = 0\) for \(j=1,\ldots,d\); in particular, the \(\lambda_i\) are rational numbers with denominator the determinant of the Gram matrix, \(\det (v_i \cdot v_j)_{i,j} =: N\), implying that \(\rho_V(A^k)\) is also rank \(d\), being a subgroup of \(V\cap A^k\) of index at most \(N^d\).

  Conversely, if \(\rho_V(A^k)\) is of rank \(d\), pick \(d\) vectors in \(A^k\) whose projections form a basis for it, and let \(P\) be the subgroup of \(A^k\) generated by these \(d\) vectors. There must be a basis \(\{a_1, \ldots, a_m\}\) of \(A^k\) ---here \(m=k\) or \(m=2k\), depending on whether \(\F = \R\) or \(\C\)--- such that the first \(d\) elements have integer multiples \(\lambda_1 a_1, \ldots, \lambda_d a_d\) which form a basis for \(P\). Then, since for every \(i = d+1, \ldots, m\), since \(\rho_V(a_i) \in \langle\rho_V(\lambda_1 a_1), \ldots, \rho_V(\lambda_d a_d)\rangle\), there is a vector of the form \(a'_i := a_i - \sum_{j=1}^d \mu_{ij} a_j\) such that \(a'_i \in V^{\perp}\). Since \(V^{\perp} \cap A^k\) has rank \(m-d\), \(V \cap A^k\) must have rank \(d\). 
\end{remark}

Then we have two types of fibers with respect to \(p\):
\begin{itemize}
\item \(q^{-1}(f)\) is a torus of rank \(\dim(\ker f)\) if \(f\in X(k,n,2r)_\Q\).
\item \(q^{-1}(f)\) is an irrational torus if \(f\in X(k,n,2r)-X(k,n,2r)_\Q\).
\end{itemize}

Now the first type of fiber is a Hausdorff space, and the closure of any point in the second type of fiber is the whole fiber. We can now describe the topology of the \(\mathcal{T}_1\)-quotient and the maximal Hausdorff quotient.

\subsubsection{\(\mathcal{T}_1\)-quotient}\label{T1-quotient}

We may construct the \(\mathcal{T}_1\)-quotient of \(\widetilde{E}(k,n,2r)\) by means of identifying points \(x\sim y\) if the intersection of the closure of their \(A^k\)-orbits \(\overline{\mathcal{O}}_x\cap \overline{\mathcal{O}}_y\) is non-empty in \(E(k,n,2r)\), as long as this identification actually yields a \(\mathcal{T}_1\)-space. Running this procedure we see that for any \(f\in X(k,n,2r)-X(k,n,2r)_\Q\)  the whole fiber \(q^{-1}(f)\) will identify to a single point, as it is an indiscrete topological space. On the other hand if \(f\in X(k,n,2r)_\Q\), then \(q^{-1}(f)\) is closed in \(\widetilde{E}(k,n,2r)\) so all of its points are closed and will remain unaffected under the identification. We may describe the resulting space as follows. Consider the composite \(\sigma_0\) \[X(k,n,2r)\xrightarrow{s_0} E(k,n,2r)\to \widetilde{E}(k,n,2r)\]
 in which the first map is the zero section \(s_0(f) = (f,0)\) and the second one is the honest quotient map. Also consider the subspace
 \[\widetilde{E}(k,n,2r)_\Q := q^{-1}(X(k,n,2r)_\Q)\subset \widetilde{E}(k,n,2r)\]
 For a topological space \(X\), let us write \(X^{\mathcal{T}_1}\) for its \(\mathcal{T}_1\)-quotient. The resulting space is then the union of these two subspaces
 \[\widetilde{E}(k,n,2r)^{\mathcal{T}_1}=\sigma_0(X(k,n,2r))\cup \widetilde{E}(k,n,2r)_\Q \]
where the quotient map \(\widetilde{E}(k,n,2r)\to\widetilde{E}(k,n,2r)^{\mathcal{T}_1}\) is the one that collapses all fibers \(q^{-1}(f)\), with \( f\in X(k,n,2r)-X(k,n,2r)_\Q\), to its corresponding point in \(\sigma_0(X(k,n,2r))\).
 
 Similarly as before, every path connected component of \(\Rep(\Z^k,\widetilde\He_{2n+1}(\R))^{\mathcal{T}_1}\) is homeomorphic to one of \(\widetilde{E}(k,n,2r)^{\mathcal{T}_1}\) for some \(2r\leq k\).

\subsubsection{\(\mathcal{T}_2\)-quotient }

For a topological space \(X\), let us write \(X^{\mathcal{T}_2}\) for its maximal Hausdorff quotient.

\begin{proposition}
If \(r>0\), \(\widetilde{E}(k,n,2r)^{\mathcal{T}_2} = X(k,n,2r)\), and \(\widetilde{E}(k,n,0)^{\mathcal{T}_2} = X(k,n,0)\vee(S^1)^{\tau k}\), where \(\tau =\dim_\R(\F)\).
\end{proposition}
\begin{proof}
For \({q}\colon\widetilde{E}(k,n,2r)\to X(k,n,2r)\) let \(q^\prime\colon\widetilde{E}(k,n,2r)^{\mathcal{T}_1}\to X(k,n,2r)\) be the unique map that factors \({q}\) through the quotient map \(\widetilde{E}(k,n,2r)\to \widetilde{E}(k,n,2r)^{\mathcal{T}_1}\). Then note that any two points in distinct fibers of \({q}^\prime\) posses disjoint open neighborhoods, but no two points in a single fiber \((q^\prime)^{-1}(f)\) have disjoint open neighborhoods, except when \(f=0\), since \((q^\prime)^{-1}(0)=(S^1)^{\tau k}\).
\end{proof}

\begin{remark}
Note that up to homotopy, the \(\mathcal{T}_2\)-quotient of \(\Rep(\Z^k,\widetilde\He_{2n+1}(\F))\) is no different from the space of almost commuting elements \(B_k(\h_{2n+1}(\F),\Z))\), except at the component of the trivial representation. 
\end{remark}

\section{A filtration of \(X(k,n,0)\) and stable homotopy type of strata}

In this last section, we follow \cite{CrabbGoncalves} to give a homotopy stable decompositions of the space of matrices that span an isotropic subspace of $(\F^{2n},\omega_n)$ of a given rank. The idea is to use Crabb's equivariant version \cite{Crabb} of Miller's stable splitting of Stiefel manifolds \cite{Miller}.


\subsection{The filtration by rank}

Recall that we may see the spaces \(X(k,n,2r)\) as matrices \(2n\times k\) matrices and hence they can be stratified by rank. Consider that map
 \[\pi\colon X(k,n,2r)\to \{2r,...,\min(k,2n,n+2r)\}\] 
given by \(\pi(f) = \dim((\ker f)^\perp)\). This is a continuous map with respect to the poset topology, and yields a poset stratification for \(X(k,n,2r)\). Then for every \(2r \leq d\leq\min(k,2n,n+2r)\) the \(d\)-stratum is
\[X(k,n,2r)_d:=\pi^{-1}(d) = \{f\in X(k,n,2r)\,|\,\dim(\ker f) = k-d\}\,.\] 
It is clear why \(d\leq \min(k,2n)\), and the extra constraint in the top rank dimension is related to the fact that if \(d>2r\), the column space of each matrix in \(X(k,n,2r)_d\) also contains an isotropic subspace of \(\F^{2n}\) with respect to \(\omega_n\).

Now consider the map 
\begin{align*}
g_d\colon X(k,n,2r)_d&\to \Gr(k-d,\R^k)\\
f&\mapsto \ker f
\end{align*}
Then we may exhibit a filtration of \(E(k,n,2r)\) by considering the pullback
\[\xymatrix{
E(k,n,2r)_d\ar[r]\ar[d]_{p_d}&\gamma_k^{k-d}\ar[d]\\
X(k,n,2r)_d\ar[r]_{g_d}&\Gr(k-d,\R^k)
}\]
of the tautological bundle \(\gamma_k^{k-d}\to \Gr(k-d,\R^k)\) along \(g_d\). Explicitly 
\[E(k,n,2r)_d:=\{(f,v)\in E(k,n,2r)\,|\,\dim(\ker f) = k-d\}\,.\]
We thus have a stratification of both the moduli space \(B_k(\h_{2n+1}(\F),A)/\widetilde{\He}_{2n+1}(\F)\) via the \(E(k,n,2r)_d\) and of \(\Hom(\Z^k,{\He}_{2n+1}(\F))=B_k(\h_{2n+1}(\F),0)\) via \(X(k,n,2r)_d\times\F^k\).

For completeness, the associated filtrations are
\begin{align*}
 E_0\subset E_1\subset\cdots\subset E_{\min(k,2n,n+2r)-2r} &= E(k,n,2r)\\
 X_0\subset X_1\subset\cdots\subset X_{\min(k,2n,n+2r)-2r} &= X(k,n,2r)
\end{align*}
where each stage of the filtration is given by
\[E_i=\bigcup^{i}_{j = 0}{E(k,n,2r)_{2r+j}}\,\,\,\,\text{and}\,\,\,\,X_i=\bigcup^{i}_{j = 0}{X(k,n,2r)_{2r+j}}\,.\]

\begin{remark}
We also have the corresponding orbit spaces \(\widetilde{E}(k,n,2r)_d := E(k,n,2r)_d/A^k\), which yield a stratification of \(\Rep(\Z^k,\widetilde{\He}_{2n+1}(\F))\). For its maximal Hausdorff quotient the \(d>0\) stratum is \(\widetilde{E}(k,n,2r)_d^{\mathcal{T}_2} = X(k,n,2r)_d\), and the 0-stratum is
\[\widetilde{E}(k,n,0)_0^{\mathcal{T}_2} = (S^1)^{\tau k}\]
where as before \(\tau = \dim_\R(\F)\).
\end{remark}

\subsection{Isotropic isometric frame extends to symplectic isometric frame}

To use Miller's splitting, we need a correspondence between isotropic frames in $(\F^{2n},\omega_n)$ and orthogonal frames in $D(\F)^{n}$. We achieve this by restricting to isotropic frames which are also orthogonal, which we will see are in correspondence with orthogonal symplectic frames.

 Instead of considering matrices \(M\) such that \(M^T\Omega_n M =0\),  it will be simpler to replace \(\Omega_n\) by the standard symplectic matrix \(J_n\), which gives a homeomorphic space of matrices \(M\). So let \(\nu_n\) be the symplectic form on \(\F^{2n}\) defined by \(J_n\), and for any vector space \(V\) over \(\F\) such that \(\dim(V)\leq n\), let us denote the space of isotropic embeddings by
\[\I(2n,V,\F)=\{\phi\in\Emb(V,\F^{2n})\,|\,\phi^*\nu_n = 0\}\,.\]
 If further \(V\) has an inner product we will also consider the subspace of orthogonal isotropic embeddings 
\[\Or\I(2n,V,\F)=\{ \phi\in\I(2n,V,\F))\,|\,\phi^*\phi = id\}\,.\]
When \(V=\F^r\) we will simply write \(\I(2n,r,\F)\) and \(\Or\I(2n,r,\F)\). 

Now recall that for every vector space \(V\), the product \(V\oplus V^*\) is a symplectic space with the standard symplectic form \(((x,\xi),(y,\eta)) \mapsto \eta(x) - \xi(y)\). We let \(\Sp(2n,V,\F)\) be the space of symplectic embeddings \(V\oplus V^*\to \F^{2n}\) with respect to the standard symplectic forms.  Similarly let \(\Or\Sp(2n,V,\F)\) be the subspace of orthogonal symplectic embeddings.

\begin{lemma}\label{lem-isotextendssymp}
Whenever \(V\) is an inner vector space over \(\F\) with \(\dim V\leq n\), the map
\begin{align*}
\Or\I(2n,V,\F)&\xrightarrow{\cong} \Or\Sp(2n,V,\F)\\
\phi&\mapsto(\begin{matrix}J_n\phi & \phi\end{matrix})
\end{align*}
is a diffeomorphism, where we have borrowed the matrix notation on the right hand side to denote the map 
\[(\begin{matrix}J_n\phi & \phi\end{matrix})(v_1,v_2^*) = J_n\phi(v_1) + \phi(v_2)\,,\] 
and \(v^*:=\langle v, -\rangle\colon V\to \F\).
\end{lemma}

\begin{proof}
First we show that \((\begin{matrix} J_n\phi &\phi\end{matrix})\) preserves the standard symplectic form over \(V\oplus V^*\). Evaluating we get  
\begin{align*}
\left(\begin{matrix} -\phi^TJ_n\\  \phi^T \end{matrix}\right)J_n(\begin{matrix} J_n\phi & \phi \end{matrix}) &= 
\left(\begin{matrix} \phi^TJ_n\phi & \phi^T\phi\\ -\phi^T\phi & \phi^TJ_n\phi\end{matrix}\right) = J_{\dim(V)}\,,
\end{align*}
as \(\phi\) preserves the zero form and orthogonal (hermitian). Since \(J_n^2 = -Id\), it follows that \((\begin{matrix} J_n\phi &\phi\end{matrix})\) is isometric. Clearly it is injective. To see that the assignment is surjective, suppose \(\varphi\colon V\oplus V^*\to \F^{2n}\) is an isometric symplectic embedding and let us write \(\varphi = (\begin{matrix} \varphi_1 & \varphi_2 \end{matrix})\). Then since \(\varphi\) preserves \(J_{\dim(V)}\), a similar calculation as above shows that \(\varphi_1 = J_n\varphi_2\).
\end{proof}

\begin{remark}
Let \(O(V)\) denote the isometric automorphisms of \(V\). The standard right \(O(V)\)-action over \(\Or\I(2n,V,\F)\) under the diffeomorphism corresponds to the action of the diagonal \(O(V)\) subgroup of \(O(V\oplus V^*)\cap\Sp(V\oplus V^*)\).
\end{remark}

\begin{lemma}\label{polar-dec-isot}
The polar decomposition for rectangular matrices restricts to a diffeomorphism 
\begin{align*}
\Or\I(2n,V,\F)\times \mathrm{Sym}(V)&\xrightarrow{\cong}\I(2n,V,\F)\\ 
(\phi,s)&\mapsto\phi\circ\exp(s)\nonumber
\end{align*}
where \(\mathrm{Sym}(V)\) is the vector space of self adjoint operators of \(V\). 
\end{lemma}
\begin{proof}
Indeed, before restricting to isotropic embeddings, this is a diffeomorphism (see for example \cite[Proposition 2.1]{CrabbGoncalves}). Now, if in a polar decomposition \(A=UP\) where \(A\) is an isotropic embedding and \(U\) has orthonormal columns, it follows that \(U\) is an isotropic embedding, as well, since any invertible matrix preserves the zero form. 
\end{proof}

\begin{remark}
We should point out that the non-compact manifolds \(\I(2n,2r,\F)\) and \(\Sp(2n,2r,\F)\) are not diffeomorphic, even though their respective isometric versions are (Lemma \ref{lem-isotextendssymp}). One may verify this for example by computing the dimensions of their respective tangent spaces. Nonetheless, we may use the above polar decomposition (Lemma \ref{polar-dec-isot}) to extend the map of Lemma \ref{lem-isotextendssymp} to
\begin{align*}
\I(2n,V,\F)&\to\Sp(2n,V,\F)\\
\phi\circ\exp(s)&\mapsto(\begin{matrix}J_n\phi\circ\exp(s) & \phi\circ\exp(-s)\end{matrix})
\end{align*}
which one can verify that it is actually an embedding.
\end{remark}

\subsection{Stable splitting of \(X(k,n,0)_d\)}

We are now ready to exhibit the stable splitting, but first we summarize Crabb and Gon\c{c}alves construction.

Miller \cite[Theorem C]{Miller} proved a stable splitting for the one-point compactification of the space of isometric embeddings between two real, complex or quaternionic vector spaces, and Crabb \cite[Theorem 1.16]{Crabb} pointed out that the construction can be done equivariantly with respect to a fairly large symmetry group. Then Crabb and Gon\c{c}alves \cite{CrabbGoncalves} used this equivariant stable splitting to get a stable splitting for the space of linear maps with kernel of a fixed dimension.

To state the form of Miller's splitting they used, we need some notation for Thom spaces and a couple of standard representations. For a representation \(\rho\colon G\to O(W)\) and a principal \(G\)-bundle \(P\to X\), we let \(X^\rho\) denote the Thom space of the associated vector bundle  \(P\times_G W\to X\), namely \(P \times_G D(W) / P \times_G S(W)\) where \(D(W)\) and \(S(W)\) denote the unit disk and unit spheres in \(W\). Now let \(\ad_m\) denote the adjoint representation of \(O(\F^m)\), where we let \(\F\) also stand for the quaternions and \(O(\mathbb{H}^m) = \Sp(m)\), and let \(\can_m\) be the canonical representation of \(O(\F^m)\) over \(\Hom(\F^m,\F)\).

Miller's stable splitting is a map
\[\bigvee_{0 \le l \le d}\Gr(l, \F^d)^{\ad_l\oplus((\F^s\oplus\F^{m-d})\otimes\can_l)} \to V_d(\F^{s+m})_+\,,\]
where \(\F^{s}\) and \(\F^{m-d}\) are trivial \(O(\F^l)\) representations. 
Crabb showed the map is equivariant for a group he denoted \(\Gamma(\F^d, \F^{m-d})\), but for what follows we will only need \(O(\F^d) \times O(\F^{m-d})\)-equivariance, which comes from a homomorphism \(O(\F^d) \times O(\F^{m-d}) \to \Gamma(\F^d, \F^{m-d})\). Here we won't describe the full symmetry group or its actions, but will only describe the action of \(O(\F^d) \times O(\F^{m-d})\). On the left-hand side it is the obvious action, but on the right it is non-obvious: an element \((u,v) \in O(\F^d) \times O(\F^{m-d})\) acts on a \(\phi : \F^{d} \to \F^{s+m}\) in \(V_d(\F^{s+m})\) via \[(u,v) \cdot \phi := (1_s \oplus u \oplus v) \circ \phi \circ u^{-1}.\]

We can use this equivariance to get an equivalence of fiber bundles over \(\Gr(d,\F^m)\) from Miller's stable splitting: consider \(O(\F^m)\) as a principle \(O(\F^d) \times O(\F^{m-d})\)-bundle over the Grassmannian and apply \(O(\F^m)\times_{O(\F^d) \times O(\F^{m-d})} (-)\) to the above map. The result is Crabb and Gon\c{c}alves' splitting:
\[{\bigvee_{0 \le l \le d}\!\!\!\!\!{}_{B}\,\,\Gr(l, \zeta^{\perp}})_B^{\ad_l\oplus((\F^s\oplus\zeta)\otimes\can_l)} \to \Or(\zeta^{\perp},\F^{s+m})_B^{+}\,,\]
where: \(B\) abbreviates \(\Gr(d, \F^m)\), \(\zeta\) is the canonical bundle of \(d\)-planes over \(B\) and \(\zeta^{\perp}\) is its orthogonal complement inside the trivial bundle of rank \(m\); \(\Gr(l, \zeta^{\perp})\) and \(\Or(\zeta^{\perp}, \F^{s+m})\) are natural bundle-extensions of the Grassmannian and the space of orthogonal embeddings, producing fiber bundles over \(B\); and, the subscripts \(B\) denote constructions  for spaces over \(B\):
\begin{itemize}
\item For spaces \(X_i \to B\) equipped with a section \(B \to X_i\), the \(B\)-wedge \(\bigvee_B X_i\) is the result of identifying the sections over \(B\) in the disjoint union \(\coprod X_i\).
\item For a representation \(\rho : G \to O(W)\) and a principal \(G\)-bundle \(P \to X\) where \(X \to B\) is a space over \(B\),  the \(B\)-Thom space \(X_B^{\rho}\) is the pushout \(B \longleftarrow P \times_G S(W) \to P \times_G D(W)\). Note that this is a space over \(B\) with fibers \((X_B^{\rho})_b \cong X_b^{\rho}\) given by the Thom space for the \(G\)-bundle \(P_b \to X_b\) of fibers over \(b \in B\). Also note this \(B\)-Thom space comes with a canonical section over \(B\) from the pushout square defining it.
\item \(X_B^{+} \) is fiberwise one-point compactification, producing a space with a section over \(B\) consisting of the points at infinity from each fiber.
\end{itemize}

Finally, Crabb and Gon\c{c}alves point out that collapsing the section over \(B\) on both sides gives a stable splitting
\[\bigvee_{0 \le l \le d}\,\,\Gr(l,d,\F^m)^{\ad_l\oplus((\F^s\oplus\zeta_k)\otimes\can_l)} \to L_{m-d}(\F^m, \F^{s+m})_+\,,\]
where \(\Gr(l,d,\F^m)\) is the total space of the bundle \(\Gr(l, \zeta^\perp)\), that is, the space of pairs of mutually orthogonal subspaces of \(\F^m\) with dimensions \(l\) and \(d\); and \(L_{m-d}(\F^m, \F^{s+m})\) is the space of linear maps with kernel of dimension \(m-d\) (here we omit the explanation of why the right-hand side has that homotopy type).

\medskip

This concludes our summary of Crabb and Gon\c{c}alves' splitting, and now we proceed to adapt their method to our situation. Recall that for every \(d\leq \min(k,n)\) we have a smooth fiber bundle \(g_d\colon X(k,n,0)_d\to\Gr(k-d,\F^k)\), and given an inner product on \(\F^k\), we may identify the fibers as
\[g_d^{-1}(V) = \I(2n,{V^\perp},\F)\,.\]
Then, up to diffeomorphism, \(g_d\) is the smooth fiber bundle \[V_d(\F^k)\times _{O(\F^d)}\I(2n,d,\F)\to \Gr(d,\F^k)\cong \Gr(k-d,\F^k)\,,\] where \(O(\R^m)=\O(m)\) and \(O(\C^m) = \U(m)\). Explicitly the diffeomorphism is given by taking a pair \((\psi,\phi)\) of a \(\psi : \F^d \to \F^k\) in \(V_d(\F^k)\) and a \(\phi : \F^d \to \F^{2n}\) in \(\I(2n,d,\F)\) to the linear map \(\F^{k} \to \F^{2n}\) which is zero on the orthogonal complement of \(\im(\psi)\) and is given by \(\phi \circ \psi^{-1}\) on \(\im(\psi)\).

Now, we pass to a fiberwise equivalent bundle with compact fibers. By Lemma \ref{polar-dec-isot} the inclusion \(\Or\I(2n,d,\F)\hookrightarrow\I(2n,d,\F)\) admits a deformation retraction, and by uniqueness of polar forms, \(O(\F^d)\) acts trivially over the \(\Sym(\F^d)\) factor of \(\I(2n,d,\F)\). Hence we have an \(O(\F^d)\)-equivariant deformation retraction, which induces a fiber homotopy equivalence between  \(g_d\) and the bundle 
\[V_{d}(\F^k)\times_{O(\F^{d})}\Or\I(2n,d,\F) \to  \Gr({d},\F^k)\,.\]
By Lemma \ref{lem-isotextendssymp} and Remark \ref{rem-orframes=orsympframes}, we have a diffeomorphism \[\Or\I(2n,d,\F)\cong V_d(D(\F)^n)\,,\]
where recall that we set \(D(\R) = \C\) and \(D(\C) = \mathbb{H}\). Therefore, \(g_d\) is fiberwise homotopy equivalent to the bundle
\[V_{d}(\F^k)\times_{O(\F^d)}V_d(D(\F)^n) \to  \Gr(d,\F^k)\,.\]

In order to apply Miller's splitting for \(V_d(D(\F)^n)\) here, we will need to rewrite this fiber bundle using the non-obvious group action that Crabb used.

\begin{lemma}
  The following fiber bundles over \(\Gr(d, \F^k)\) are isomorphic:
  \[V_{d}(\F^k)\times_{O(\F^d)}V_q(D(\F)^n) \to  \Gr(d,\F^k)\, \quad\text{and}\quad O(\F^k) \times_{O(\F^d) \times O(\F^{k-d})} V_d(D(\F)^n) \to \Gr(d, \F^k),\]
  where \((u,v) \in O(\F^d) \times O(\F^{k-d})\) acts on \(\phi \in V_d(D(\F)^n)\), which we view as an isometric embedding \[\phi \colon D(\F)^d \to D(\F)^{n-d} \oplus D(\F)^d \oplus D(\F)^{k-d},\] via \((u,v) \cdot \phi := (1_{n-d} \oplus u \oplus v) \circ \phi \circ u^{-1}\) ---where we implicitly extend scalars from \(\F\) to \(D(\F)\).
\end{lemma}

\begin{proof}
  The  isomorphisms are given by:
  \begin{align*}
    [(\psi,\phi)] & \mapsto [(\eta,(1\oplus\eta^{-1})\circ\phi)], \quad\text{and}\\
    [(\eta,\phi)] & \mapsto [(\eta|_{\F^d}, (1\oplus \eta)\circ\phi)],
  \end{align*}
  where in the first map \(\eta \in O(\F^k)\) is an arbitrarily chosen extension of \(\psi : \F^d \to \F^k\) to an orthogonal map on all of \(\F^k\). It is straightforward to check the maps are well-defined, mutually inverse and commute with the projections to the Grassmannian.
\end{proof}

Finally, we can apply the functor \(O(\F^k) \times_{O(\F^d) \times O(\F^{k-d})}(-)\) to Miller's equivariant splitting for \(V_d(D(\F)^n)_+\), and then collapse the base sections on either side, to obtain a splitting for the stratum \(X(k,n,0)_d\), as recorded in the following proposition. (Notice that the splitting is equivariant with respect to \(O(D(\F)^d) \times O(D(\F)^{k-d})\), but we only use the subgroup \(O(\F^d) \times O(\F^{k-d})\).)

\begin{theorem}\label{prop:stable-splitting}
 For every \(0\leq d \le k \le n\), the \(d\)-stratum \(X(k,n,0)_d\) with a disjoint basepoint splits stably as a wedge of Thom spaces
\[\bigvee_{0\leq l\leq d}[O(\F^k)\times_{O(\F^d)\times O(\F^{k-d})}\Gr(l,D(\F)^d)]^{\ad_l\oplus(D(\F)^{n-d}\oplus\zeta)\otimes\can_l}\,.\] 
\end{theorem}

{
 {\sc
  Instituto de Matem\'aticas, UNAM, Mexico City, Mexico}\\ 
  \emph{E-mail address}:
  {\texttt{omar@matem.unam.mx}}}\\
  
  {
  {\sc
  Centro de Investigaci\'on en Matem\'aticas, Unidad M\'erida, Yucat\'an, Mexico}\\ 
  \emph{E-mail address}:
  {\texttt{bernardo.villarreal@cimat.mx}}}


\begin{thebibliography}{80}



\bibitem{AdemCohenGomezMPCPS} A. Adem, F. Cohen and J. M. G\'{o}mez. Stable splittings, spaces of representations and almost commuting elements in Lie groups. \emph{Math. Proc. Cambridge Philos. Soc.} \textbf{149} (2010), 455-490. 

\bibitem{AdemCohenTorres} A. Adem, F. Cohen and E. Torres. Commuting elements, simplicial spaces and filtrations of classifying spaces. \emph{Math. Proc. Cambridge Philos. Soc.} \textbf{152} (2012), 1, 91-114. 

\bibitem{AdemGomez} A. Adem and J. M. G\'{o}mez. A classifying space for commutativity in Lie groups.  \emph{Algeb. Geom. Topol.} \textbf{15} (2015) 493-535. 

\bibitem{A-CGV21} O. Antol\'{i}n-Camarena, S. Gritschacher and B. Villarreal. {Higher generation by abelian subgroups in Lie groups}. \emph{Transform. Groups} {\bf 28} (2023), 1375--1390.



\bibitem{Crabb} M. C. Crabb. On the stable splitting of \(U(n)\) and \(\Omega U(n)\). \emph{Algebraic Topology Barcelona 1986}, Springer Lecture Notes in Math. {\bf 1298}, 1987, 35--53.


\bibitem{CrabbGoncalves} M. C. Crabb and D. L Gon\c{c}alves. On the space of matrices of a given rank. \emph{Proc. Edin. Math. Soc.}, {\bf 32}  (1989), 99--105. 

\bibitem{DemazureGabriel} M. Demazure et P. Gabriel, \emph{Groupes algébriques}, Tome 1, Masson-North Holland Pub. Cy, Paris, 1970.

\bibitem{GoldmanFlat2} W. M. Goldman. Flat bundles with solvable holonomy II: Obstruction theory. \emph{Proc. Amer. Math. Soc.} {\bf 83} (1981), no. 1, 175–178.



\bibitem{KnappLieGroups} A. W. Knapp. \emph{Lie Groups: Beyond an Introduction}, Digital Second Edition. Published by the Author,
East Setauket, New York, 2023.

\bibitem{Miller} H. Miller. Stable splitting of Stiefel manifolds. \emph{Topology}, no. 4, {\bf 24} (1985), 411--419.

\bibitem{P-S} A. Pettet and J. Souto. Commuting tuples in reductive groups and their maximal compact subgroups. \emph{Geom. Topol.}, {\bf 17} (2013), no. 5, 2513--2593.

 
\end{thebibliography}
\end{document}